\documentclass[11pt]{amsart}
\usepackage{amsthm}
\usepackage{amsmath}
\usepackage{amscd}
\usepackage{t1enc}
\usepackage{indentfirst}
\usepackage{graphicx}
\usepackage{amsrefs}
\usepackage{graphics}
\numberwithin{equation}{section}
\usepackage[margin=2.9cm]{geometry}
\usepackage{epstopdf} 
\usepackage{tikz-cd}
\usepackage{tikz}
\usetikzlibrary{matrix, arrows}
\usepackage{wrapfig}
\usepackage{amsfonts}
\usepackage{xcolor}
\usepackage{mathtools}
\usepackage{setspace}
\usepackage{fancyhdr}
\usepackage{float}
\usetikzlibrary{decorations.markings}

\theoremstyle{plain}
\newtheorem{Th}{Theorem}[section]

 \theoremstyle{definition}

\newtheorem{?}[Th]{Problem}

\newtheorem*{Ex(s)}{Spherical Example}
\newtheorem{Ex(t)}{Toroidal Example}
\newtheorem*{fact*}{Fact}
\newtheorem{Case}{Case}

\begin{document}
\title[Symplectic Structures with Non-Isomorphic Primitive Cohomology]{Symplectic Structures with Non-Isomorphic Primitive Cohomology on open 4-Manifolds}
\author{Matthew Gibson}
\address{\parbox{\linewidth}{Department of Mathematics, University of California,
Irvine, CA 92697, USA}}
\email{gibsonmd@uci.edu}
\author{Li-Sheng Tseng}
\address{\parbox{\linewidth}{Department of Mathematics, University of California,
Irvine, CA 92697, USA}}
\email{lstseng@math.uci.edu}
\author{Stefano Vidussi}
\address{\parbox{\linewidth}{Department of Mathematics, University of California,
Riverside, CA 92521, USA}}
\email{svidussi@ucr.edu}
\begin{abstract} We analyze four-dimensional symplectic manifolds of type $X=S^1 \times M^3$ where $M^3$ is an open $3$-manifold admitting inequivalent fibrations leading to inequivalent symplectic structures on $X$.  For the case where $M^3 \subset S^3$ is the complement of a $4$-component link constructed by McMullen-Taubes, we provide a general algorithm for computing the monodromy of the fibrations explicitly. We use this algorithm to show that certain inequivalent symplectic structures are distinguished by the dimensions of the primitive cohomologies of differential forms on $X$.  We also calculate the primitive cohomologies on $X$ for a class of open $3$-manifolds that are complements of a family of fibered graph links in $S^3$.  In this case, we show that there exist pairs of symplectic forms on $X$, arising from either equivalent or inequivalent pairs of fibrations on the link complement, that have different dimensions of the primitive cohomologies.

\end{abstract}
\maketitle
%%%%%%%%%%%%%%%%%%%%%%%%%%%%%%%%%%%%%%%%%%%%%%%%%%%%%%%%%%%%%%%%%%%%%%%%%%%%%%%%%%%%%%%%%%%%%%%
%%%%%%%%%%%%%%%%%%%%%%%%%%%%%%%%%%%%%%%%%%%INTRODUCTION%%%%%%%%%%%%%%%%%%%%%%%%%%%%%%%%%%%%%%%%%%%%
%%%%%%%%%%%%%%%%%%%%%%%%%%%%%%%%%%%%%%%%%%%%%%%%%%5%%%%%%%%%%%%%%%%%%%%%%%%%%%%%%%%%%%%%%%%%%%%
\section{Introduction}
\par There is a well-known class of symplectic 4-manifolds that is intrinsically related to Riemann surfaces and its mapping class group.  Given a finite type Riemann surface $\Sigma$ and an element $f:\Sigma \to \Sigma$ of its mapping class group, we can combine the information of $(\Sigma, f)$ into a three-dimensional manifold, the mapping torus $\pi : Y_f \to S^1$ with fiber $\Sigma$ and monodromy given by $f$.  We then take the product of $Y_f$ with a circle $S^1$ and obtain a symplectic 4-manifold $X=S^1 \times Y_f$ whose symplectic structure can be constructed as follows (see Section \ref{sec-prelim} for more detail).  Let $d\pi$ be the closed, non-degenerate 1-form associated to the projection map and $dt$ the 1-form of the $S^1$ factor in $X$. Together with a global two-form $\omega_\Sigma$ of $Y_f$, which restricts to the volume form on each fiber, we can take the symplectic form to be $\omega_f = dt\wedge d\pi+\omega_\Sigma$.  By construction, the data of this symplectic 4-manifold $(X, \omega_f)$ consists essentially of just $(\Sigma, f)$.  

Though simple in construction, this class of 4-manifolds has an interesting property.   We recall that a three-dimensional manifold can have different fibrations with homeomorphic fibers, with the monodromy given by different elements of the mapping class group, or even fibrations with non-homeomorphic fibers.  Of interest, the different fibrations would lead to different symplectic forms on $X$.  It is therefore an interesting question whether symplectic forms from different fibrations are related or not in some particular way.  In fact, it is known from the work of McMullen and Taubes \cite{McT} that the symplectic structures from different fibrations can be {\it inequivalent}. In this paper, following \cite{McT}, we say that two symplectic forms $(\omega_0, \omega_1)$ are {\it equivalent} if there is some combination of diffeomorphisms and smooth paths of symplectic forms interpolating between $\omega_0$ and $\omega_1$. Otherwise, we will call two symplectic forms {\it inequivalent}.

These 4-manifolds of course can be studied using tools from symplectic geometry.  From this perspective, it is interesting to ask whether symplectic techniques on 4-manifolds can shed some light on the relationships between elements of the mapping class group of Riemann surfaces, in the case of different fibrations with homeomorphic fibers, and their associated symplectic structures.  McMullen and Taubes used Seiberg-Witten invariants to demonstrate the inequivalence of two symplectic forms in their construction.  Here, we will consider a more algebraic tool related to cohomologies of primitive differential forms introduced by Tseng and Yau \cite{TsengII}.

In \cite{TsengII}, Tseng-Yau studied a new cohomological invariant for any symplectic manifold $(X, \omega)$ denoted by $PH_\pm^*(X,\omega)$ (see also \cite{TsengIII} for its underlying algebraic structure). Unlike the de Rham cohomology, the dimension of this symplectic cohomology is not an invariant of homeomorphism type and can depend on the symplectic structure. In fact, it was demonstrated explicitly in a 6-nilmanifold example that $PH_\pm^*(X,\omega)$ can jump in dimensions as the symplectic structure $\omega$ on $X$ varies \cite{TsengII}.  But prior to this work, as far as the authors are aware, no examples of dimension jumping were known for symplectic 4-manifolds.  We will explain below how the dimensions of $PH_\pm^2(X,\omega)$ effectively count the number of Jordan blocks of size at least two in the decomposition of $f^*-1: H^1(\Sigma)\to H^1(\Sigma)$. This key fact together with our explicit calculations of the monodromy allow us to distinguish different symplectic structures on 4-manifolds of type $X=S^1\times M^3$ arising from different fibrations for $M^3$.  As a first example, inspired by the construction of McMullen and Taubes in \cite{McT}, we find the following. 
\begin{Th}\label{thm-A}
There exists a fibered 3-manifold $M^3$ admitting two different fibrations with monodromy $f$ and $g$, yielding two associated symplectic structures on $X = S^1 \times M^3$ -- $(X, \omega_f)$ and $(X, \omega_{g})$ -- that can be distinguished by the primitive cohomologies. In particular, 
\begin{align*}
\dim PH^2_+(X,\omega_f)\neq \dim PH^2_+(X,\omega_{g}).
\end{align*}
\end{Th} 
%
%Furthermore, as stated in Theorem \ref{thm-A}, these symplectic structures are {\it inequivalent}. 
The first example of inequivalent symplectic structures with distinct primitive cohomologies that we will describe in this paper is closely related to the construction of McMullen and Taubes where they showed the existence of a pair of inequivalent symplectic forms on a simply-connected 4-manifold starting from a 3-manifold $M^3$ that is a fibered link complement \cite{McT}. Their construction furnishes for us the fibration of $M^3$ with monodromy $f$ referenced in Theorem \ref{thm-A}. In \cite{Vidussi}, the third author found another fibration of $M^3$ with monodromy $g$ whose associated symplectic structure is inequivalent from that associated with McMullen-Taubes' fibration.  After carefully describing the monodromies $f$ and $g$, we are able to show that the symplectic 4-manifolds $(X, \omega_f)$ and $(X, \omega_{g})$ have distinct primitive cohomologies.

\par It turns out that our first example motivated by McMullen and Taubes is just the tip of the iceberg in terms of inequivalent symplectic structures having different primitive cohomologies. A large class of examples can be constructed from fibered $3$-manifolds that are $S^3$ complements of graph links and studied in detail by Eisenbud and Neumann \cite{Eisenbud}. In \cite{Graph_Link}, the third author constructed a family of 3-manifolds, denoted by $M^{(2n)}=S^3\backslash K^{(2n)}$ for any $n\in \mathbb{N}$, where $K^{(2n)} \subset S^3$ is a 2-component graph link. Its fibrations are encoded by a pair of integers $(m_1, m_2)$ satisfying certain conditions. It was shown in \cite{Graph_Link} that the associated symplectic $4$-manifold $X^{(n)}=S^1\times M^{(2n)}$ admits at least $n+1$ inequivalent symplectic structures. We have investigated the primitive cohomology of this family of $4$-manifolds and found that the dimension depends on the choice of the fibration $(m_1, m_2)$, and in fact, directly relates to the divisibility of the $m_i$'s by 3.
\begin{Th}\label{thm-B}
For $n\in \mathbb{N}$, let $(m_1,m_2)$ be coprime, representing a fibration of the graph link $M^{(2n)}=S^3\backslash K^{(2n)}$. By reversing the roles of $m_1$ and $m_2$ if necessary, we write $m_1= 3^kq$ with $\gcd(q,3)=1$ and assume $\gcd(3,m_2)=1$. Denote by $\omega_{(m_1, m_2)}$ the associated symplectic form on $X^{(n)}=S^1\times M^{(2n)}$.  The primitive cohomology is given by 
\begin{align*}
\dim PH_2^+(X^{(n)}, \omega_{(m_1, m_2)})=\left\{ \begin{array}{lc}b_2(X^{(n)})+ 3^{n+k-\lceil \frac{k}{2}\rceil}-3^k, &k\leq 2n-4\\b_2(X^{(n)})+2\cdot 3^k, &2n-3\leq k\leq 2n-2\\b_2(X^{(n)})+1,& k\geq 2n-1 \end{array}\right.
\end{align*}
where $b_2(X^{(n)}) = 3$ is the second Betti number.
\end{Th}
\par These 4-manifolds associated with graph links not only have primitive cohomologies that may vary with the symplectic structures but also  demonstrate other interesting properties worthy of further investigations. For instance, in addition to inequivalent symplectic structures possibly having different primitive cohomologies, we also found pairs of symplectic structures associated with {\it equivalent} 3-manifold fibrations having different primitive cohomologies as well. Similar to symplectic forms, we say for fibered 3-manifolds that two fibrations $(\pi_0, \pi_1)$ of $M^3\to S^1$ are {\it equivalent} if their corresponding one-forms $(d\pi_0, d\pi_1)$ are related by some combinations of diffeomorphisms and smooth paths of one-forms interpolating between them. Clearly, to distinguish equivalent versus inequivalent 3-manifold fibrations from the symplectic 4-manifold perspective would require more tools than just the dimension of the primitive cohomology.  Regarding this, it may be helpful to point out that the differential forms underlying $PH^*(X,\omega)$ are situated in an $A_3$-algebra \cite{TsengIII} and hence the cohomology has a ring structure and also Massey products different from those of $H^*(X)$. For the class of symplectic 4-manifolds $X=S^1 \times M^3$, where $M^3$ is the mapping torus with monodromy $f$, it would be insightful to relate the invariants coming from $PH^*(X, \omega_f)$ and all its products directly with the properties of the monodromy $f$ as an element of the mapping class group.  This can give a deeper link between the space of symplectic structures on $X$ with the mapping class groups of Riemann surfaces and their corresponding mapping tori. This paper represents a significant first step in this direction focusing on the implications of the cohomology group.
\par The structure of the paper is as follows. In Section \ref{sec-prelim}, we review the basic properties of $PH^*(X,\omega)$ and the de Rham cohomology of a fibered 3-manifold. In Section \ref{examples}, we discuss the details of the McMullen-Taubes fibrations \cite{McT} and use their work and that of \cite{Vidussi} to establish Theorem \ref{thm-A}. The explanation of our calculations is given in Section \ref{sec-construct} and the Appendix, where we provide the subtle details of our explicit construction of the monodromy of the McMullen-Taubes surface bundle. In Section \ref{sec-GL}, we prove Theorem \ref{thm-B}, where we investigate the symplectic structures and the primitive cohomologies of 4-manifolds constructed from the family of fibrations on graph links $K^{(2n)}$ introduced by the third author in \cite{Graph_Link}. These fibrations are determined by relatively prime pairs $(m_1, m_2)$.  Finally, we end in Section \ref{sec-D} with a discussion of our results.
\subsection*{Acknowledgments}
We would like to express our thanks to Yi Liu, Curtis McMullen, Yi Ni, Nick Salter, Chung-Jun Tsai, and Jesse Wolfson for helpful conversations.  Additionally, we thank the support of the Simons Collaboration Grant No. 636284 (L.S.T.) and No. 524230 (S.V.).  
%%%%%%%%%%%%%%%%%%%%%%%%%%%%%%%%%%%%%%%%%%%%%%%%%%5%%%%%%%%%%%%%%%%%%%%%%%%%%%%%%%%%%%%%%%%%%%%
%%%%%%%%%%%%%%%%%%%%%%%%%%%%%%%%%%%%%%%%%%%PRELIMINARIES%%%%%%%%%%%%%%%%%%%%%%%%%%%%%%%%%%%%%%%%%%%%
%%%%%%%%%%%%%%%%%%%%%%%%%%%%%%%%%%%%%%%%%%%%%%%%%%5%%%%%%%%%%%%%%%%%%%%%%%%%%%%%%%%%%%%%%%%%%%%
\section{Preliminaries}\label{sec-prelim}
%%%%%%%%%%%%%%%%%%COHOMOLOGIES%%%%%%%%%%%%%%%%%%%%
\subsection{De Rham and Primitive Cohomologies}\mbox{} \\
\par In this subsection, we briefly review the basics of the de Rham cohomology of surface bundles over a circle and also summarize the primitive cohomology studied in \cite{TsengIII,TsengII}, applying it to symplectic 4-manifolds associated to surface bundles. 
\par Let ${\widehat \Sigma_g}$ be a closed surface of genus $g$.
We endow ${\widehat \Sigma}_g$ a symplectic structure, and we can assume, up to isotopy, that any self-diffeomorphism ${\widehat f} \colon {\widehat \Sigma}_g \to {\widehat \Sigma}_g$ preserves the symplectic form. Next, we consider the mapping torus of ${\widehat f}$, the closed 3-manifold
${\widehat Y}_{\widehat f} ={\widehat \Sigma}_g \times[0,1]/ (x,1)\sim ({\widehat f}(x),0)$. It follows that ${\widehat Y}_{\widehat f}$ has a ${\widehat \Sigma}_g$-bundle structure over $S^1$ with the projection given by $\pi:{\widehat Y}_{\widehat f} \to S^1$, $\pi([x, t]) = t$. The associated map ${\widehat f}$ is called the {\it monodromy} of the bundle. We can endow the 4-manifold $S^1 \times {\widehat Y}_{\widehat f}$ with a symplectic structure by defining ${\widehat \omega} := dt \wedge d\pi + {\widehat \omega}_{\widehat \Sigma}$ where $dt$ is the volume form on the  $S^1$ factor and ${\widehat \omega}_{\widehat \Sigma}$ is a 2-form on ${\widehat Y}_{\widehat f}$ restricting to the volume form on ${\widehat \Sigma}$. 

Let $\Sigma_{g,n}={\widehat \Sigma}_g-\{y_1,\cdots, y_n\}$ be the surface of genus $g$ with $n$ points removed. When clear, the surface will simply be abbreviated by $\Sigma$. Moreover, when convenient, $P:=\{y_1,\cdots, y_n\}$ may be thought of as marked points. The symplectic form on ${\widehat \Sigma}_g$ restricts to a symplectic form on $\Sigma$.   

Let ${\widehat f} \colon {\widehat \Sigma}_g \to {\widehat \Sigma}_g$  be any symplectic diffeomorphism preserving $P$ setwise. This restricts to a symplectic diffeomorphism $f \colon \Sigma \to \Sigma$. The 3-dimensional mapping torus $Y_f =\Sigma\times[0,1]/ (x,1)\sim (f(x),0)$ is an open submanifold of ${\widehat Y}_{\widehat f}$, obtained by removing  a collection of curves transverse to the fibers. Denoting, with slight abuse of notation, the new fibration by $\pi \colon Y_f \to S^1$ as well,  its  de Rham cohomology is determined by the  Wang exact sequence
\begin{equation*}
\begin{tikzcd}
\cdots\arrow{r}&H^0(\Sigma)\arrow{r}&H^1(Y_f)\arrow{r}&H^1(\Sigma)\arrow{r}{f^*-1}&H^1(\Sigma)\arrow{r}&H^2(Y_f)\arrow{r}&\cdots
\end{tikzcd}
\end{equation*}
This sequence yields
\begin{align*}
H^0(Y_f)&= \mathbb{R},\\
H^1(Y_f)&=\ker(f^*-1: H^1(\Sigma) \to H^1(\Sigma))\oplus \langle d\pi\rangle,\\
H^2(Y_f)&=\langle d\pi\rangle\wedge\text{coker}(f^*-1: H^1(\Sigma)\to H^1(\Sigma)),\\
H^3(Y_f)&= 0.
\end{align*}

For $X=S^1\times Y_f$, the restriction of the symplectic form on $S^1 \times {\widehat Y}_{\widehat f}$ to $X$ 
%$S^1\times Y_f$ 
yields the symplectic form $\omega = dt\wedge d\pi + \omega_\Sigma$, where $\omega_\Sigma$ (by abuse of notation) is a closed 2-form on $Y_f$ that restricts to the symplectic form on each fiber.  The Kunneth formula also easily shows
\begin{align*}
H^0(X)&=\mathbb{R},\\
H^1(X)&=\langle dt, d\pi\rangle \oplus \ker(f^*-1: H^1(\Sigma) \to H^1(\Sigma)),\\
H^2(X)&=\langle dt\wedge d\pi\rangle \oplus d\pi\wedge\text{coker}(f^*-1: H^1(\Sigma)\to H^1(\Sigma))\oplus dt\wedge \ker(f^*-1: H^1(\Sigma) \to H^1(\Sigma)),\\
H^3(X)&=\langle dt\wedge d\pi\rangle\wedge \text{coker}(f^*-1: H^1(\Sigma)\to H^1(\Sigma)),\\
H^4(X)&=0.
\end{align*}
%
%A peculiar property on the open manifolds that we will consider in what follows is that they arise by applying the constructions above simultaneously to different mapping tori ${\widehat Y}_{\widehat f}$, originating from different monodromies or even surfaces of different genus.

\par We now turn to describing the primitive cohomology on the symplectic 4-manifold.  Given a symplectic manifold $(X^{2n}, \omega)$, choose a basis $\{\partial_{x^i}\}_{i=1}^{2n}$ for $TX$. Its differential forms $\Omega^*(X)$ carry an $sl_2(\mathbb{R})$ action, with the following  $sl_2$-representation:
\begin{align*}
L &:\Omega^k(X)\to \Omega^{k+2}(X)\\
& A_k\mapsto \omega\wedge A_k\\
\\
\Lambda &:\Omega^k(X)\to \Omega^{k-2}(X)\\
& A_k\mapsto \frac{1}{2}(\omega^{-1})^{ij}\iota_{\partial_{x^i}}\iota_{\partial_{x^j}}A_k\\
\\
H &:\Omega^k(X)\to \Omega^k(X)\\
&A_k\mapsto (n-k)A_k
\end{align*}
with $[H,\Lambda] = 2 \Lambda$, $[H, L]=-2L$, and $[\Lambda, L]= H$.
\par The {\it primitive forms} $\mathcal{P}^*(X)$ are the highest weight vectors in this algebra. That is, $A_k\in \mathcal{P}^k(X)$ precisely if $\Lambda A_k = 0 = \omega^{n-k+1}\wedge A_k$. This action leads to the Lefschetz decomposition so that any $k$-form has an expression $A_k = B_k + \omega\wedge B_{k-2}+\omega^2\wedge B_{k-4} +\cdots$ where each $B_i$ is primitive. In \cite{TsengII}, Tseng and Yau constructed the differentials $\partial_{\pm} : \mathcal{P}(X)^k \to \mathcal{P}^{k\pm1}(X)$ and the elliptic complex
\[
  \begin{CD}
     0 @>>>  \mathcal{P}^0  @>{\partial_+}>>  \mathcal{P}^1  @>{\partial_+}>>  \mathcal{P}^2  @>{\partial_+}>>  \cdots @>{\partial_+}>>\mathcal{P}^n\\
     @. @.@.@. @. @VV{\partial_+\partial_-}V\\
     0 @<{\partial_-}<<\mathcal{P}^0@<{\partial_-}<<\mathcal{P}^1@<{\partial_-}<<\mathcal{P}^2@<{\partial_-}<<\cdots@<{\partial_-}<<\mathcal{P}^n
  \end{CD}
\]
whose ``top'' and ``bottom'' cohomologies are denoted $PH^*_{+}(X,\omega)$ and  $PH^*_{-}(X,\omega)$, respectively.
Also in \cite{TsengIII}, it is proven that certain de Rham cohomological data is enough to compute the symplectic cohomology groups $PH^*_{\pm}(X)$, given by the isomorphisms below for $k\leq n$:
\begin{align}
PH^k_+(X) &\cong \text{coker}(L: H^{k-2}(X) \to H^k(X))\oplus \ker(L: H^{k-1}(X)\to H^{k+1}(X)),\label{SCI}\\
PH^k_-(X) &\cong \text{coker}(L: H^{2n-k-1}(X)\to H^{2n-k+1}(X))\oplus ker(L: H^{2n-k}(X)\to H^{2n-k+2}(X)). \label{SCII}
\end{align}
\par Let us first discuss the case where $\omega$ is chosen so that $[\omega]_{dR} = [dt\wedge d\pi]_{dR}$, the more general case will be treated at the end of the section. Applying equations (\ref{SCI}) and (\ref{SCII}) to the 4-manifold $X=S^1\times Y_f$,  along with computations from earlier in this section, yield
\begin{align*}
PH^0_+(X)&\cong \mathbb{R},\\
PH^1_+(X)&\cong H^1(X),\\
PH^2_+(X) &\cong H^2(X)/\langle dt\wedge d\pi\rangle \oplus \langle dt, d\pi\rangle\oplus [\ker(f^*-1)\cap \text{Im}(f^*-1)],\\
PH^2_-(X)&\cong H^2(X)\oplus \left[\langle dt\wedge d\pi\rangle\wedge \text{coker}(f^*-1)\right]/\left[\langle dt\wedge d\pi\rangle\wedge \ker(f^*-1)\right],\\
PH^1_-(X)&\cong H^3(X),\\
PH^0_-(X)&\cong 0.
\end{align*}
Let $b_i$ denote the Betti numbers of $X$ and ${\it p}_i^{\pm}(X,\omega)$ denote the dimensions of $PH^i_{\pm}(X,\omega)$. When the choice of the underlying symplectic structure is clear, we simply write ${\it p}_i^{\pm}$. Then,
\begin{align*}
{\it p}_0^+&=1,\\
{\it p}_1^+&=b_1,\\
{\it p}_2^+&= b_2+1+\text{dim}\left[\ker(f^*-1)\cap\text{Im}(f^*-1)\right],\\
{\it p}_2^-&= b_2+\text{dim}\left[\ker(f^*-1)\cap\text{Im}(f^*-1)\right],\\
{\it p}_1^-&= b_3,\\
{\it p}_0^-&=0,
\end{align*}
where we have used the fact that $\text{dim}\left[\ker(f^*-1)\cap\text{Im}(f^*-1)\right]$ and \newline $\text{dim}\left[(dt\wedge d\pi\wedge \text{coker}(f^*-1))/(dt\wedge d\pi\wedge \ker(f^*-1))\right]$ are equal by realizing that both quantities count the number of Jordan blocks of $f^*-1$ of size strictly greater than 1 (see discussion below). We note that the {\it primitive Euler characteristic} $\chi_p(X) =\sum (-1)^i{\it p}_i^+-\sum (-1)^i{\it p}_i^-= 2-b_1+b_3$ is fixed under homeomorphism type. However, the primitive Betti numbers ${\it p}_2^{\pm}$ may vary in general. 
\par Let us explain how this dimension relates to the Jordan blocks of $f^*-1$. For brevity we write $\nu_2 :=\text{dim}\left[\ker(f^*-1)\cap\text{Im}(f^*-1)\right]$. Now, if $\alpha \in \ker(f^*-1)\cap\text{Im}(f^*-1)$, then $(f^*-1)\alpha = 0$ and $(f^*-1)\beta = \alpha$ for some $\beta$. That is, $\alpha$ is an eigenvector in a Jordan chain of length at least 2. It follows that $\nu_2$ counts the number of Jordan blocks corresponding to eigenvalue $\lambda =1$ of size {\it at least} 2. More generally, there is a descending filtration of subgroups $PH^2_+(M)\supset J_1(M)\supset J_2(M)\supset\cdots$ where $J_k(M) = \ker(f^*-1)\cap\text{Im}(f^*-1)^k$. If $\alpha \in J_k(M)$, then it is the eigenvector in a Jordan chain of length at least $k+1$ given by $x_1 = \alpha$, $x_2 = (f^*-1)^{k-1}\beta$, $x_3 = (f^*-1)^{k-2}\beta$,$\cdots$, $x_k = (f^*-1)\beta$, $x_{k+1} = \beta$. Thus, the dimension of the filtered quotient $J_{k-1}/J_{k}$ counts the number of Jordan blocks of size {\it exactly} $k$.
\par In general, it is also possible to consider other symplectic structures on $X=S^1 \times Y_f$ where $[\omega]\neq [dt\wedge d\pi]$. Though not our main focus, we will describe how a different choice of the symplectic structure would affect the calculation of the primitive cohomology in the remainder of this subsection.  To begin, let  $i: \Sigma \xhookrightarrow{} Y_f$ be the inclusion map of the fiber and choose $\tilde{\omega}_f\in \Omega^2(Y_f)$ such that $i^*(\tilde{\omega}_f) = \omega_\Sigma$. Furthermore, assume $\tilde{\omega}_f$ can be chosen so that $[\omega_0] := [dt\wedge d\pi + \tilde{\omega}_f] = [dt\wedge d\pi]$. Then $PH^*(X,\omega_0)$ is given by the above computations. Now, given any 1-form $\eta \in \Omega^1(Y_f)$ such that $d(\eta\wedge d\pi)=0$, we can define a new symplectic form, $\omega_\eta:= \omega_0 + \eta\wedge d\pi= (dt+\eta)\wedge d\pi + \tilde{\omega}_f$. Of interest are those $\eta$'s such that $[\omega_\eta]\neq [\omega_0]$, which occurs precisely when $[d\pi\wedge \eta]\in H^2(Y_f)$ is non-trivial. Choose a Jordan basis $\{x_{i,0}\}_{i=1}^k$ for $\ker(f^*-1)$ and denote the corresponding Jordan chain of $x_{i,0}$ by $\{x_{i,0}, x_{i,1},\cdots, x_{i,n_i}\}$. Rearranging if necessary, we assume $n_i=0$ for $1\leq i \leq s$. Thus $\{x_{i,0}\}_{i=1}^s$ are the Jordan blocks of size exactly 1. Then, we can write
\begin{align*}
H^1(Y_f)&= \langle d\pi \rangle \oplus \langle x_{i,0}\rangle_{i=1}^k,\\
H^2(Y_f)&= \langle d\pi \wedge x_{i,n_i}\rangle_{i=1}^k,
\end{align*}
and express $[d\pi\wedge \eta]= \sum_{i=1}^k\lambda_i[d\pi\wedge x_{i,n_i}]$. We may write $PH^2_+(X,\omega_\eta)=H^2(X)/\langle[\omega_\eta]\rangle \oplus K_\eta$ where $K_\eta=\ker(\omega_\eta\wedge: H^1(X)\to H^3(X))$. Then
\begin{align*}
[\omega_\eta\wedge d\pi] &=[0],\\
[\omega_\eta\wedge dt]&= [\eta\wedge d\pi \wedge dt]=-[dt\wedge d\pi\wedge \eta],\\
[\omega_\eta\wedge x_{i,0}]&= [dt\wedge d\pi\wedge x_{i,0}].
\end{align*}
\par We see that $[\omega_\eta\wedge(\sum_{i=1}^s\lambda_ix_{i,0}+dt)] =[dt\wedge d\pi\wedge\sum_{i=s+1}^k\lambda_ix_{i,n_i}]$, which is trivial if and only if $\eta \in\ker(f^*-1)$. Similarly, denote by $C_\eta= \text{coker}(\omega_\eta\wedge: H^1(X)\to H^3(X))$. The above computations show $C_\eta \cong \langle dt\wedge d\pi\wedge x_{i,n_i}\rangle_{i={s+1}}^k/\langle dt\wedge d\pi\wedge \eta\rangle$. The quotient by the $\eta$ term will be extraneous in the case that $\eta\in \ker(f^*-1)$. In all, the resulting primitive cohomology groups $PH^*(X, \omega_\eta)$ can be expressed as follows.
\begin{align*}
PH^0_+(X,\omega_\eta)&\cong H^0(X),\\
PH^1_+(X,\omega_\eta)&\cong H^1(X),\\
PH^2_+(X,\omega_\eta)&\cong H^2(X)/\langle [\omega_\eta]\rangle \oplus K_\eta, \\
PH^2_-(X,\omega_\eta)&\cong H^2(X)\oplus C_\eta,\\
PH^1_-(X,\omega_\eta)&\cong  H^3(X),\\
PH^0_-(X,\omega_\eta)&\cong \langle 0 \rangle,
\end{align*} 
where 
\begin{align*}
K_\eta\cong& \begin{cases} \langle d\pi\rangle\oplus \langle x_{i,0} \rangle_{i=s+1}^k, &\lambda_i \neq 0 \text{\ for \ some \ }i>s\\ \langle d\pi, dt+\eta\rangle\oplus \langle x_{i,0} \rangle_{i=s+1}^k,&\lambda_i = 0 \text{\ for \ all \ }i>s  \end{cases}\\
\\
C_\eta\cong& \begin{cases}  \langle dt\wedge d\pi\wedge x_{i,n_i}\rangle_{i={s+1}}^k/\langle dt\wedge d\pi\wedge \eta\rangle, &\lambda_i \neq 0 \text{\ for \ some \ }i>s\\  \langle dt\wedge d\pi\wedge x_{i,n_i}\rangle_{i={s+1}}^k,&\lambda_i = 0 \text{\ for \ all \ }i>s  \end{cases}
 \end{align*}
\par Regardless of the class of $\eta$, we see $PH^k_\pm(X,\omega_\eta)$ are isomorphic to de Rham cohomologies for $k=0,1$. Furthermore, in the case that $\eta$ descends to a cohomology class $[\eta]\in H^1(Y_f)$, the above computations show $\dim PH^*(X,\omega_\eta)=\dim PH^*(X, \omega_0)$. Unless otherwise stated, in the remainder of the paper, we assume $[\omega] = [dt\wedge d\pi]$.
 %%%%%%%%%%%%%%MAPPING CLASS GROUP%%%%%%%%%%%%%%%%%%%%
\subsection{Mapping Class Group}\label{Mapping Torus}\mbox{}\\
\par In this subsection, we review some of the necessary topics relevant to our work from mapping class group theory. We focus mainly on the mapping class group of $\Sigma_{1,4}$, detailing a set of generators described by Birman in \cite{Birman}. We wish to study the diffeomorphisms of $\Sigma_{g,n}$ up to an equivalence. We define the {\it mapping class group}, denoted by $\mathcal{M}(\Sigma_{g,n})$, as the group of diffeomorphisms fixing $P$ setwise, up to isotopies fixing $P$ setwise. We define the {\it pure mapping class group}, $\mathcal{P}\mathcal{M}(\Sigma_{g,n})$, as the subset of elements from $\mathcal{M}(\Sigma_{g,n})$ fixing $P$ pointwise. Since the majority of this paper takes place in $\mathcal{P}\mathcal{M}(\Sigma_{1,4})$ we briefly discuss the diffeomorphisms generating this subgroup for the torus with four marked points. We define $\tau_i$ as the longitudinal curve which passes above $y_1, y_2,\cdots, y_{i-1}$, through $y_i,$ and below $y_{i+1},\cdots, y_n$.  Denote by $\rho_i$ the meridian curve passing through $y_i$.
\par From these curves we define homeomorphisms $\mathcal{P}{ush}(\tau_i)$ and $\mathcal{P}{ush}(\rho_i)$, called the point-pushing maps. These are classical maps in mapping class group theory. They may be loosely visualized as follows: $\mathcal{P}{ush}(\tau_i)$ is the map which pushes the point $x_i$ around the curve $\tau_i$, ``dragging'' the rest of the surface $\Sigma_{1,4}$ with it. $\mathcal{P}{ush}(\rho_i)$ has a similar interpretation. In \cite{Birman}, Birman showed that the push maps generate the mapping class group:
$$\mathcal{P}\mathcal{M}(\Sigma_{1,4}) = \langle \mathcal{P}{ush}(\tau_i), \mathcal{P}{ush}(\rho_i)\rangle, i=1,2,3,4.$$ 
\par It turns out that these maps can be realized in terms of Dehn twists along homology generators for $H_1(\Sigma_{1,4})$. These explicit expressions are worked out in the Appendix. (The curves $\rho_i$ and $\tau_i$ are pictured in Figure \ref{rho_and_tau}, drawn on the square representing $\Sigma_{1,4}$.)
 \begin{figure}
 \caption{$\rho_i$ and $\tau_i$ paths on $\Sigma_{1,4}$}
 \begin{center}
 \tikzset{-<-/.style={decoration={
  markings,
  mark=at position .37 with {\arrow{<}}},postaction={decorate}}}
  \tikzset{->-/.style={decoration={
  markings,
  mark=at position .6 with {\arrow{>}}},postaction={decorate}}}
 \begin{tikzpicture}[scale =2]
 \draw (0,0) grid (1,1);
\fill (canvas cs: x=0.1cm, y =0.8cm) circle (1pt);
\fill (canvas cs: x=0.37cm, y =0.65cm) circle (1pt);
\fill (canvas cs: x=0.64cm, y =0.5cm) circle (1pt);
\fill(canvas cs: x=0.85cm, y=0.35cm) circle (1pt);
\draw[->-, blue, thick] (0.1,0)--(0.1,1);
\draw[->-, blue, thick] (0.37,0)--(0.37,1);
\draw[->-, blue, thick] (0.64,0)--(0.64,1);
\draw[->-, blue, thick] (0.85,0)--(0.85,1);
\node at (0.14,-0.1) {$\rho_1$};
\node at (0.35,-0.1) {$\rho_2$};
\node at (0.60,-0.1) {$\rho_3$};
\node at (0.86,-0.1) {$\rho_4$};
\node at (0.2,0.82) {$1$};
\node at (0.46,0.72) {$2$};
\node at (0.72,0.65) {$3$};
\node at (0.95,0.46) {$4$};
 \draw (2,0) grid (3,1);
\draw[->-, red, thick] (3,0.8)--(2,0.8);
\draw[->-, red, thick] (3,0.65)--(2,0.65);
\draw[->-, red, thick] (3,0.5)--(2,0.5);
\draw[->-, red, thick] (3,0.35)--(2,0.35);
\fill (canvas cs: x=2.1cm, y =0.8cm) circle (1pt);
\fill (canvas cs: x=2.37cm, y =0.65cm) circle (1pt);
\fill (canvas cs: x=2.64cm, y =0.5cm) circle (1pt);
\fill(canvas cs: x=2.85cm, y=0.35cm) circle (1pt);
\node at (3.12,0.83) {$\tau_1$};
\node at (3.12,0.65) {$\tau_2$};
\node at (3.12,0.47) {$\tau_3$};
\node at (3.12,0.32) {$\tau_4$};
 \end{tikzpicture}
 \label{rho_and_tau}
 \end{center}
 \end{figure}
\par Another important subgroup of the mapping class group is the {\it Torelli group}, $\mathcal{I}(\Sigma)$, consisting of diffeomorphisms acting trivially on (co)homology. Thus, $$\mathcal{I}(\Sigma) = \{f\in \mathcal{M}(\Sigma):1= f^*: H^1(\Sigma)\to H^1(\Sigma)\}.$$ \par Calculations in Section \ref{Mapping Torus} show that if $f\in \mathcal{I}(\Sigma)$, then $H^*( S^1\times Y_f) = H^*(T^2\times\Sigma)$ and $PH^*(S^1\times Y_f) = PH^*(T^2\times\Sigma)$ as groups. Thus two Torelli-bundles cannot be distinguished from their primitive cohomology groups alone. However, by the same reasoning, $f\in \mathcal{I}$ and $g\not\in \mathcal{I}$ can {\it always} be distinguished by the dimension of the cohomology groups.
%%%%%%%%%%%%%%%%%%%%%%%%%%%%%%%%%%%%%%%%%%%%%%%%%%5%%%%%%%%%%%%%%%%%%%%%%%%%%%%%%%%%%%%%%%%%%%%
%%%%%%%%%%%%%%%%%%%%%%%%%%%%%%%%%%%%%%%%MCMULLEN TYPE 4-MANIFOLDS%%%%%%%%%%%%%%%%%%%%%%%%%%%%%%%%%%%%%%%%
%%%%%%%%%%%%%%%%%%%%%%%%%%%%%%%%%%%%%%%%%%%%%%%%%%%%%%%%%%%%%%%%%%%%%%%%%%%%%%%%%%%%%%%%%%%%%%
\section{McMullen-Taubes Type 4-manifolds} \label{examples}
\par In this section, we will discuss different presentations of a $3$--manifold, the complement of a link in $S^3$, as fibration with fiber a punctured torus or sphere. All the torus fiber examples will induce symplectic structures with identical primitive cohomologies but the sphere fibration will be shown to give primitive cohomology of different dimension. 

We quickly review the examples constructed in \cite{McT} and \cite{Vidussi}. In \cite{McT}, McMullen and Taubes considered a 3-manifold $M$ which is a link complement $S^3\backslash K$. Here, $K$ is the Borromean rings $K_1\cup K_2\cup K_3$ plus $K_4$, the axis of symmetry of the rings. By performing 0-surgery along the Borromean rings we obtain a presentation of $M$ as $\mathbb{T}^3\backslash L$ where:
\begin{itemize}
\item $L \subset \mathbb{T}^3$ is a union of four disjoint, closed geodesics $L_1, L_2, L_3, L_4$,
\item $H_1(\mathbb{T}^3)=\langle L_1, L_2. L_3\rangle$,
\item $L_4 = L_1+L_2+L_3$.
\end{itemize}
\par The fiber of $M$ is the 2-torus with punctures coming from the $L_i$. The different fibration structures are captured by the Thurston ball. In \cite{McT}, this ball is computed as the dual of the Newton polytope of the Alexander polynomial. Endow the ball with coordinates $\phi=(x,y,z,t)$ as in \cite{McT}. Then, the Thurston unit ball has 16 top--dimensional faces (each fibered) coming in 8 pairs under the symmetry $(\phi, -\phi)$. Furthermore, restricting to faces that are dual to those vertices of the Newton polytope with no $t$--component, we get 14 faces, that come in two types; quadrilateral and triangular. It is shown in \cite{McT} that there exists a pair of inequivalent symplectic forms on a 4-manifold coming from different fibrations of $\mathbb{T}^3\backslash L$. These fibrations correspond to points lying on the two distinct types of faces. In \cite{Vidussi}, it is shown that the remaining pair of $16-14=2$ faces (with a non-zero $t$-component) yield a third symplectic structure which is inequivalent to the 
two found by McMullen and Taubes. 
\par We will investigate the monodromy of the fibration given in \cite{Vidussi}, in which it is observed that $M$ admits a fibration with fiber the four-punctured 2-sphere. Table \ref{mon-table} summarizes the conclusions of the examples to follow. Determining these monodromy formulas explicitly is a crucial step in computing the dimension of $PH_{\pm}^2(X, \omega)$, since it depends on the Jordan decomposition.
\par The first example is the fibration with fiber $\Sigma_{0,4}$, hence `spherical' type.  The other two examples are of `toroidal' type with fiber $\Sigma_{1,4}$. In the spherical example, the given projection vector is the cohomology class in $H^1(M^3)$ corresponding to a point on the Thurston ball. The  projection vectors of the `toroidal' type examples refer to the vector used in its fiber bundle construction and not the point on the Thurston ball. These details are elaborated on in the Appendix.
\begin{table}
\caption{Monodromies} \label{mon-table}
\begin{center}
  \begin{tabular}{ | l | c |r| }
    \hline
    Type & Projection Vector $v_1$ &Monodromy \\ \hline
       Spherical& (0,0,0,1)& $\sigma_1^{-1}\sigma_2\sigma_1^{-1}\sigma_2\sigma_1^{-1}\sigma_2$\\\hline
    Toroidal& $(-1,-1,1)$&$\tau_3^{-1}\tau_2^{-1}\tau_1^{-1}\rho_1^{-1}\rho_2^{-1}\tau_1^{-1}\rho_2\tau_4^{-1}\rho_4^{-1}\tau_3^{-1}$\\ \hline
    Toroidal& $(-1,1,1)$&$\rho_2^{-1}\tau_1\rho_2^{-1}\tau_1^{-1}\tau_4^{-1}\rho_3^{-2}\tau_2^{-1}\rho_4^{-1}\rho_1^{-1}$  \\ \hline
  \end{tabular}
\end{center}
\end{table}
For notational simplicity, in Table \ref{mon-table}, $\mathcal{P}{ush}(\rho_i)$ and  $\mathcal{P}{ush}(\tau_i)$ are abbreviated to $\rho_i$ and $\tau_i$, respectively.
%%%%%%%%%%%%%%%%%%%%%% EX1: Vidussi%%%%%%%%%%%%%%%%%%%%
\begin{Ex(s)}
In this example, we take the fibration from \cite{Vidussi} obtained by performing 0-surgery along the $K_4$ axis. The fiber is $S^2$ punctured four-times, with monodromy given by the braid word corresponding to the Borromean rings.  Let $\sigma_i$ denote the half-Dehn twist which switches marked points $i$ and $i+1$. This homeomorphism can be viewed similar to the push map, where we ``push'' the surface through the arc connecting the $i$th and $(i+1)$th points. As a braid, it is the element which passes the $i$th string over the $(i+1)$th string. Under this identification, the monodromy is given by $$\sigma_1^{-1}\sigma_2\sigma_1^{-1}\sigma_2\sigma_1^{-1}\sigma_2.$$
\end{Ex(s)}
\par The derivation of the toroidal type monodromies is much more involved. We carefully work out these formulas in the next section. For now, we take the monodromies from Table \ref{mon-table} as true and examine their cohomological implications.
\subsection*{Cohomological Analysis}
\par Denote by $g$ the monodromy from fiber the four--punctured 2-sphere $\Sigma_{0,4}$. Similarly, $f$ denotes either of the two monodromies coming from the four--punctured torus fiber $\Sigma_{1,4}$ in Table \ref{mon-table}. With the monodromy $f$, we can compute its action on $H^1(\Sigma_{1,4})$ (either by hand or with the help of software) to conclude that $\text{dim}\ker(f^*-1) =b_1(Y_{f})-1= 3$ in both `toroidal' cases. Let $X = S^1\times Y_{f} = S^1\times Y_g$, these manifolds being diffeomorphic by the above discussion. We will compute the primitive cohomology of the symplectic structures associated to the fibrations determined by the monodromies $f$ and $g$. 
\par With respect to the ordering $(a_0, a_1, a_2, a_3, b_0)$ of basis vectors for $H^1(\Sigma_{1,4})$,  computation shows the action on $H^1(\Sigma_{1,4})$ is given by 
\[f^*-1=
\begin{pmatrix}
-1&-1&-1&-1&1\\
0&0&0&0&-1\\
1&1&1&1&1\\
0&0&0&0&-1\\
0&0&0&0&0
\end{pmatrix},
\hspace{0.3in}
J= \begin{pmatrix}
0&1&0&0&0\\
0&0&0&0&0\\
0&0&0&1&0\\
0&0&0&0&0\\
0&0&0&0&0
\end{pmatrix},
\]
for all $f$. Here $J$ is the Jordan matrix for $f^*-1$. We note it has two blocks of size 2 and one of size 1. It follows that
\begin{align*}
\ker(f^*-1)&=\langle (1,0,0,-1,0), (0,1,0,-1,0), (0,0,1,-1,0)\rangle,\\
\text{Im}(f^*-1)&= \langle (-1,0,1,0,0), (1,-1,1,-1,0) \rangle.\\
\end{align*}
A quick check shows 
\begin{align*}
(f^*-1)(-1,0,1,0,0) &= 0=(f^*-1)(1,-1,1,-1,0).
\end{align*}
Hence, we conclude $$\text{dim}\ker(f^*-1)\cap\text{Im}(f^*-1) =\dim\text{Im}(f^*-1)= 2.$$ Notice this dimension agrees with the number of blocks from $J$ of size at least 2. Computations from Section 2.1 show that for any choice of $\eta$ associated to $\omega_f$ we have either
\begin{align*}
{\it p}_2^{+}(X,\omega_f)=& \begin{cases} 9, & \eta \text{ is such that }\lambda_i \neq 0 \text{\ for some\ }i>s\\ 10,&\eta \text{ is such that }\lambda_i = 0 \text{\ for all\ }i>s  \end{cases}
\end{align*}
We now turn to $(X, \omega_g)$. Since $Y_f$ is  diffeomorphic to $Y_g$, we must have 
$$ \text{dim}\ker(g^*-1) =\dim\ker(f^*-1)= 3.$$ Moreover, using the formula $\chi(\Sigma_{g,n}) = 2-2g-n$, it follows $\chi(\Sigma_{0,4})=-2 = 1 - b_1(\Sigma_{0,4})$, and so $b_1(\Sigma_{0,4}) = 3$. But by Rank-Nullity, $3 = 3 + \dim\text{Im}(g^*-1)$, from which it follows $\dim\ker(g^*-1)\cap\text{Im}(g^*-1) = 0$.  Thus ${\it p}_2^+(X,\omega_g) = b_2(X_g) +1 =8\neq {\it p}_2^+(X,\omega_f)$. 
\par We point out that from the Jordan form of the $f$, these monodromies are not Torelli elements of $\mathcal{M}(\Sigma_{1,4})$. However, by dimension considerations, we saw $\dim\text{Im}(g^*-1) =0$ and so $g$ {\it is} a Torelli element of $\mathcal{M}(\Sigma_{0,4})$. Moreover, even though each $f$, $f'$  coming from fiber $\Sigma_{1,4}$ are not Torelli, $f^*=f'^*$ and so it follows that $f'f^{-1}$ is a Torelli element.
\newline
\par These calculations give the following theorem.
\begin{Th}\label{fibration-thm}
There exist inequivalent fibrations of the $3$-manifold $M$ with inequivalent associated symplectic 4-manifolds $(X, \omega_f)$, $(X, \omega_g)$, which can be distinguished by primitive cohomologies. In particular, $$p_2^+(X,\omega_f)\neq p_2^+(X,\omega_g).$$
\end{Th} 
To establish Theorem \ref{fibration-thm}, it only remains to verify the toroidal type monodromies in Table \ref{mon-table}. Before this verification, we point out an interesting feature of $p_2^+$ in this situation. By calculations from Section 2, $p_2^+(X, \omega_h) = 8 +\dim\ker(h^*-1)\cap\text{Im}(h^*-1) \geq 8$ for {\it any} monodromy $h$. However, if $h$ is some monodromy on $\Sigma_{1, 4}$ we know its Jordan decomposition must give blocks of size $(2, 2, 1)$ or $(3, 1, 1)$. That is, $h$ has at least one block of size of at least 2, so that $p_2^+ > 8$. Consequently, $p_2^+$ can distinguish the monodromy type in the sense that $p_2^+(X, \omega_f)$ is always larger than $p_2^+(X, \omega_g)$ for any $f \in \mathcal{M}(\Sigma_{1, 4})$ and $g \in \mathcal{M}(\Sigma_{0, 4})$. Moreover, the toroidal type examples above all have block sizes $(2, 2, 1)$. But if a monodromy with sizes $(3, 1, 1)$ existed, $p_2^+$ could also distinguish the two. Nonetheless, the exploration of the explicit constructions of the McMullen-Taubes type monodromies is still valuable. It provides an algorithm to produce the concrete Dehn twists which are responsible for the actions of $f^*-1$ on cohomology.
%
%%%%%%%%%%%%%%%%%%%%%%%%%%%%%%%%%%%%%%%%%%%%%%%%%%%5%%%%%%%%%%%%%%%%%%%%%%%%%%%%%%%%%%%%%%%%%%%
%%%%%%%%%%%%%%%%%%%%%%%%%%%%%%%%%%%%%%%CONSTRUCTION OF MONODROMIES%%%%%%%%%%%%%%%%%%%%%%%%%%%%%%%%%%%%%%%
%%%%%%%%%%%%%%%%%%%%%%%%%%%%%%%%%%%%%%%%%%%%%%%%%%%%%%%%%%%%%%%%%%%%%%%%%%%%%%%%%%%%%%%%%%%%%%
\section{Construction of Monodromies}\label{sec-construct}
In this section, we provide details for the construction of the toroidal monodromies in Table \ref{mon-table}. The Appendix gives an even more specific outline of the procedure that follows. In the examples to come, we take different bases $v_1 = (a_1, a_2, a_3)$, $v_2=(1,1,0)$, $v_3=(0,1,1)$ and fiber along $v_1$ so that the fiber at time $t$ looks like $\Sigma_{t,4} = tv_1+\langle v_2, v_3\rangle$ with marked points
\begin{align*}
y_1(t)&=(-4\epsilon, 3\epsilon) + (a_3-a_2,-a_3)t,\\
y_2(t)&=(-\epsilon, 2\epsilon) + (-a_1, a_1-a_2)t,\\
y_3(t)&=(0,0)+(a_3-a_2,a_1-a_2)t,\\
y_4(t)&=(\epsilon,-3\epsilon)+(-a_1,-a_3)t.\\
\end{align*}
\par Here, $\epsilon$ is some small fixed constant used to shift the marked points away from the origin at $t=0$. The vector $v_1$ is the projection vector given in column 2 of Table \ref{mon-table}. The general idea is as follows,
\begin{enumerate}
\item Using the paths of the punctures $y_i$, find relative locations to determine if $y_i$ passes above or below $y_j$.
\item  Express $\mathcal{P}{ush}(y_i(t))$ of the $y_i$ path in terms of generators $\mathcal{P}{ush}(\rho_i)$, $\mathcal{P}{ush}(\tau_i)$.
\item Calculate the intersection points of punctures $(y_i(t), y_j(t))$ at times $(t_i, t_j)$. If $t_i>t_j$ then $y_i$ crosses over $y_j$. If $t_i< t_j$ then $y_j$ crosses over $y_i$.
\item Use the crossings information to determine the order of $\mathcal{P}{ush}(y_i(t))$ maps in the final monodromy.
\end{enumerate}
\par The procedure is best demonstrated through examples. As before, we drop the push notation so that $\mathcal{P}ush(\rho_2)\mathcal{P}ush(\tau_1)^{-1}\mathcal{P}ush(\tau_3)$ is simply denoted by $\rho_2\tau_1^{-1}\tau_3$. We also use function notation right to left so that the previous word indicates $y_3$ travels along $\tau_3$ then $y_1$ along the inverse of $\tau_1$ then finally $y_2$ along $\rho_2$. Homeomorphism type of the below examples was confirmed with SnapPy \cite{SnapPy}.
%
%%%%%%%%%%%%%%%%%%%%%%% TOROIDAL EX1: v1=(-1,-1,1)%%%%%%%%%%%%%%%%%%%%
\begin{Ex(t)} $v_1= (-1,-1,1)$ 
\newline
The paths of the corresponding marked points are
\begin{align*}
y_1(t)&=(-4\epsilon, 3\epsilon) + (2,-1)t,\\
y_2(t)&=(-\epsilon, 2\epsilon) + (1, 0)t,\\
y_3(t)&=(0,0)+(2,0)t,\\
y_4(t)&=(\epsilon,-3\epsilon)+(1,-1)t.\\
\end{align*}
 \par Thus $y_2$ and $y_3$ travel in a parallel horizontal direction. $y_1$ and $y_4$ travel downwards and to the right and so will intersect both $y_2$ and $y_3$. We first find these intersection times. We illustrate the process for $y_1$ and $y_3$ and summarize the other points in Table \ref{Ex2_Ints}. We need times $t_1$ and $t_3$ so that $y_1(t_1) = y_3(t_3)$. In other words, we seek a solution to the system
 \begin{align*}
-4\epsilon+2t_1=2t_3,\\
3\epsilon-t_1=0,
 \end{align*}
which gives $(t_1, t_3) =(3\epsilon, \epsilon+\frac{n}{2})$, $n=0,1$. Hence $y_1$ and $y_3$ intersect twice. The first time $y_1$ passes over $y_3$. Then at $t_3=\epsilon+\frac{1}{2}$, $y_3$ crosses $y_1$. At $t_2=\frac{5}{8}\epsilon +\frac{1}{2}$, $y_2$ passes over $y_1$. Similarly solving the corresponding system for $y_2$ and $y_3$ yields $(t_2, t_3) = (\frac{2}{3}\epsilon + \frac{n}{2}, 1-\frac{1}{3}\epsilon)$, $n=0,1$. Both $y_2$ times occur before $y_3$, hence we conclude $y_3$ passes over $y_2$ twice. The remaining points of intersection are given in Table \ref{Ex2_Ints}. The times specified are the later of the two crossing times and the points have been listed in order of intersection occurrence, from first to last.
\begin{table}[h]
\caption{Toroidal Example 1 Intersections} \label{Ex2_Ints}
\centering
  \begin{tabular}{ | l | c |r| }
    \hline
    Points & Time &Crossing \\ \hline
    $(y_1,y_3)$ & $3\epsilon$& $y_1$ over $y_3$ \\ \hline
    $(y_1, y_3)$ & $\epsilon+\frac{1}{2}$&$y_3$ over $y_1$\\ \hline
    $(y_2,y_4)$ & $1-3\epsilon$& $y_2$ over $y_4$ \\ \hline
    $(y_3,y_4)$& $1-3\epsilon$& $y_4$ over $y_3$\\\hline
    $(y_1,y_2)$ & $1-\epsilon$& $y_2$ over $y_1$\\\hline
    $(y_1,y_4)$& $1-\epsilon$&$y_1$ over $y_4$\\ \hline
    $(y_3,y_4)$& $1-\epsilon$&$y_3$ over $y_4$\\ \hline
  \end{tabular}
\end{table}
\par Pictured in Figure \ref{example 1.1} are the paths of the $y_i$ drawn in the plane (up to identification), where we have decomposed the ``diagonal'' paths of $y_1$ and $y_4$ into a combination of basis curves $\rho_i$ and $\tau_i$. To find the path of $y_1$, for example, we must use its velocity vector $(2,-1)$ as well as the relative locations of $y_1$ with respect to the start points of $y_2$, $y_3$, and $y_4$. Given that point $y_2$ starts at $(-\epsilon, 2\epsilon)$, we have $y_1(\frac{3}{2}\epsilon)=(-\epsilon,\frac{3}{2}\epsilon)$ and so $y_1$ travels `below' the $y_2$ start point. Similar computations show $y_1$ travels above both the $y_3$ and $y_4$ start points. As illustrated in Figure \ref{example 1.1}, the velocity vector $(2,-1)$ suggests $y_1$ has a path given by $\tau_1^{-1}\rho_1^{-1}\tau_1^{-1}$. However the diagonal path homotopic to this combination will not preserve the condition that $y_1$ travels below the $y_2$ start point. To remedy this situation, we must begin the $y_1$ monodromy with the loop $C_{12}$. This curve travels counterclockwise from $y_1$, enclosing $y_2$. Figure \ref{example 1.2} illustrates the $\tau_1^{-1}C_{12}$ portion of the monodromy.
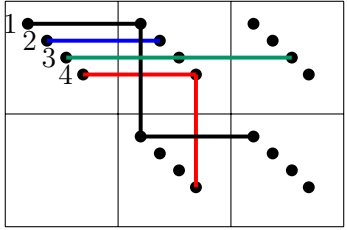
\begin{figure}[h]
\caption{Example 1 Marked Point Paths}
\centering
 \begin{tikzpicture}[scale =1.5]\label{example 1.1}
 \draw (0,0) grid (3,2);
\fill (canvas cs: x=0.2cm, y =1.8cm) circle (1.5pt);
\draw (0.2cm, 1.8cm) node[left] {1};
\fill (canvas cs: x=0.37cm, y =1.65cm) circle (1.5pt);
\draw (0.37cm, 1.65cm) node[left] {2};
\fill (canvas cs: x=0.54cm, y =1.5cm) circle (1.5pt);
\draw (0.54cm, 1.5cm) node[left] {3};
\fill(canvas cs: x=0.69cm, y=1.35cm) circle (1.5pt);
\draw (0.69cm, 1.35cm) node[left] {4};
\fill (canvas cs: x=1.2cm, y =1.8cm) circle (1.5pt);
\fill (canvas cs: x=1.37cm, y =1.65cm) circle (1.5pt);
\fill (canvas cs: x=1.54cm, y =1.5cm) circle (1.5pt);
\fill(canvas cs: x=1.69cm, y=1.35cm) circle (1.5pt);
\fill (canvas cs: x=1.2cm, y =0.8cm) circle (1.5pt);
\fill (canvas cs: x=1.37cm, y =0.65cm) circle (1.5pt);
\fill (canvas cs: x=1.54cm, y =0.5cm) circle (1.5pt);
\fill(canvas cs: x=1.69cm, y=0.35cm) circle (1.5pt);
\fill (canvas cs: x=2.2cm, y =0.8cm) circle (1.5pt);
\fill (canvas cs: x=2.37cm, y =0.65cm) circle (1.5pt);
\fill (canvas cs: x=2.54cm, y =0.5cm) circle (1.5pt);
\fill(canvas cs: x=2.69cm, y=0.35cm) circle (1.5pt);
\fill (canvas cs: x=2.2cm, y =1.8cm) circle (1.5pt);
\fill (canvas cs: x=2.37cm, y =1.65cm) circle (1.5pt);
\fill (canvas cs: x=2.54cm, y =1.5cm) circle (1.5pt);
\fill(canvas cs: x=2.69cm, y=1.35cm) circle (1.5pt);
\draw[-, line width=1.5pt, blue] (0.37,1.65)--(1.37,1.65);
\draw[-, line width =1.5pt, red] (0.69,1.35)--(1.69,1.35);
\draw[-, line width =1.5pt, red] (1.69,1.35)--(1.69,0.35);
\draw[-, line width=1.5pt] (0.2,1.8)--(1.2,1.8);
\draw[-, line width=1.5pt] (1.2,1.8)--(1.2,0.8);
\draw[-,line width=1.5pt] (1.2,0.8)--(2.2,0.8);
\draw[-, line width=1.5pt, green!60!blue] (0.54,1.5)--(1.54, 1.5);
\draw[-, line width=1.5pt, green!60!blue] (1.54,1.5)--(2.54, 1.5);
 \end{tikzpicture}
\end{figure}
\begin{figure}
\caption{$C_{12}$ Path in Example 1}
\centering
 \begin{tikzpicture}[scale =1.5]\label{example 1.2}
\tikzset{->-/.style={decoration={
markings,
mark=at position .67 with {\arrow{>}}},postaction={decorate}}}
 \draw (0,0) grid (2,1);
\fill (canvas cs: x=0.2cm, y =0.8cm) circle (1.5pt);
\draw (0.22cm, 0.9cm) node[left] {1};
\fill (canvas cs: x=0.37cm, y =0.65cm) circle (1.5pt);
\draw (0.33cm, 0.65cm) node[left] {2};
\fill (canvas cs: x=0.60cm, y =0.48cm) circle (1.5pt);
%\draw (0.64cm, 0.45cm) node[left] {3};
\fill(canvas cs: x=0.76cm, y=0.33cm) circle (1.5pt);
%\draw (0.69cm, 0.2cm) node[left] {4};
\fill (canvas cs: x=1.2cm, y =0.8cm) circle (1.5pt);
\draw[->-, line width=1pt] (0.2,0.8)--(1.2,0.8);
\draw[->-,line width=1pt] (0.2,0.8) .. controls (0.4,0.25) and (0.8,0.65) .. (0.2,0.8);
 \draw (3,0) grid (5,1);
\fill (canvas cs: x=3.2cm, y =0.8cm) circle (1.5pt);
\draw (3.22cm, 0.9cm) node[left] {1};
\fill (canvas cs: x=3.37cm, y =0.65cm) circle (1.5pt);
\draw (3.33cm, 0.65cm) node[left] {2};
\fill (canvas cs: x=3.60cm, y =0.48cm) circle (1.5pt);
%\draw (0.64cm, 0.45cm) node[left] {3};
\fill(canvas cs: x=3.76cm, y=0.33cm) circle (1.5pt);
%\draw (0.69cm, 0.2cm) node[left] {4};
\fill (canvas cs: x=4.2cm, y =0.8cm) circle (1.5pt);
\draw[->-, line width=1pt] (3.7,0.8)--(4.2,0.8);
\draw[->-,line width=1pt] (3.2,0.8) .. controls (3.3,0.42) .. (3.7,0.8);
 \end{tikzpicture}
\end{figure}
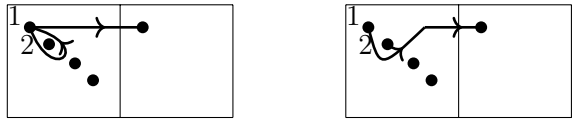
\par $y_4$ is the only other diagonal path. We can easily check that it travels above the $y_1$, $y_2$, and $y_3$ start points. Hence its path is simply given by $\tau_4^{-1}\rho_4^{-1}$, indicated by the $(1,-1)$ velocity vector. 
\par Summarizing, the monodromies of the punctures are given by
\begin{align*}
y_1(t)&:  \tau_1^{-1}\rho_1^{-1}\tau_1^{-1}C_{12}=\tau_1^{-1}\rho_1^{-1}\rho_2^{-1}\tau_1^{-1}\rho_2,\\
y_2(t)&: \tau_2^{-1}, \\
y_3(t)&: \tau_3^{-2},\\
y_4(t)&: \tau_4^{-1}\rho_4^{-1}.
\end{align*}
Now, we must determine the order of these individual monodromies in the final map. Using the above formulas, it's clear $y_2(t)$ and $y_3(t)$ are parallel so their relative order to each other in the final monodromy doesn't matter. From Table 1, we see every other point crosses over $y_3$ first, but then $y_3$ crosses over $y_1$ and $y_4$ again later. Thus we should put one $\tau_3^{-1}$ at the beginning of the monodromy and the other $\tau_3^{-1}$ at the end. Next, both $y_1$ and $y_2$ cross over $y_4$ so the $y_4$ term should come next. 
\par It only remains to determine the order of $y_1$ and $y_2$, which is given by Table 1 as $y_1$ then $y_2$. Therefore our monodromy has the formula $y_3\circ y_2\circ y_1\circ y_4\circ y_3$, where the first and last $y_3$ terms are each a $\tau_3^{-1}$. This ordering gives 10 possible crossings, but $y_2$ and $y_3$ are parallel and $y_3$ appears twice. Hence the number reduces to $10-3 = 7$, matching the occurrences in Table \ref{Ex2_Ints}. 
\par Piecing all the arguments together shows the final monodromy is isotopic to
\begin{align*}
&\tau_3^{-1}\tau_2^{-1}(\tau_1^{-1}\rho_1^{-1}\tau_1^{-1}C_{12})\tau_4^{-1}\rho_4^{-1}\tau_3^{-1}=\tau_3^{-1}\tau_2^{-1}(\tau_1^{-1}\rho_1^{-1}\rho_2^{-1}\tau_1^{-1}\rho_2)\tau_4^{-1}\rho_4^{-1}\tau_3^{-1}.
\end{align*} 
\end{Ex(t)}
\
\\
%%%%%%%%%%%%%%%%%%%%%% TOROIDAL EX2: v1=(-1,1,1)%%%%%%%%%%%%%%%%%%%%
\begin{Ex(t)} $v_1= (-1,1,1)$
\newline
The paths of the punctures are given by
 \begin{align*}
 y_1(t)&=(-4\epsilon, 3\epsilon)+(0,-1)t,\\
 y_2(t)&=(-\epsilon,2\epsilon)+(1, -2)t,\\
 y_3(t)&=(0,0)+(0,-2)t,\\
 y_4(t)&=(\epsilon, -3\epsilon) + (1, -1)t.
 \end{align*}
 Implementing the techniques from the previous example, we obtain the intersections in Table \ref{Ex3_Ints}. There is only one non-trivial diagonal path, given by $y_2$. Evaluating this path at the appropriate times yields 
 \begin{align*}
 y_2(-3\epsilon)&=(-4\epsilon,8\epsilon),\\
 y_2(\epsilon)&=(0,0),\\
 y_2(2\epsilon)&=(\epsilon,-2\epsilon).
 \end{align*} 
 We see that $y_2$ travels above $y_1$ and $y_4$ start points and through $y_3$ at the origin. We note at $t=\epsilon$,  $y_3(\epsilon)=(0,-2\epsilon)$ has traveled away from the origin and so $y_2(t)$ and $y_3(t)$ do not actually collide. Thus, in between $\rho_2^{-1}\rho_2^{-1}\tau_2^{-1}$, we must insert a loop traveling counterclockwise starting at $y_2$ and enclosing $y_1$. It turns out this curve is also homotopic to $C_{12}$ (see \cite{Birman} for more discussion). By drawing a diagram similar to Figure \ref{example 1.1} one can see the correct placement should be $\rho_2^{-1}C_{12}\rho_2^{-1}\tau_2^{-1}$. The  paths of the other points are straightforward, given by 
\begin{align*}
y_1& : \rho_1^{-1},\\
y_2& : \rho_2^{-1}C_{12}\rho_2^{-1}\tau_2^{-1}=\rho_2^{-1}\tau_1\rho_2^{-1}\tau_1^{-1}\tau_2^{-1},\\
y_3& : \rho_3^{-2}, \\
y_4& : \tau_4^{-1}\rho_4^{-1}.
\end{align*}
\par The ordering for this example is similar to that of Example 1; this time we need to split both of the paths $y_2$ and $y_4$ into two parts each. Notice from the individual monodromies that $y_1$ and $y_3$ are parallel so their relative order doesn't matter. We proceed by considering the remaining interactions separately. Since $y_1$ passes under for all its crossings, it appears first. Then $y_3$ over $y_2$ and $y_2$ over $y_4$ suggests the ordering $y_3\circ y_2\circ y_4$. However, we need $y_4$ to cross over $y_3$ and this current arrangement does the opposite. Hence we must split the $y_4$ monodromy into two components:  $y_4\circ y_3\circ y_2\circ y_4$. Finally, if we leave $y_2$ together, we will have both $y_4$ and $y_2$ crossing over one another at different times. Consequently, we also split $y_2$ for the ultimate ordering given by  $y_2\circ y_4\circ y_3\circ y_2\circ y_4\circ y_1$. The final monodromy pieces together as $$y_2\circ \tau_4^{-1}\circ \rho_3^{-2}\circ y_2\circ\rho_4^{-1}\circ\rho_1^{-1}.$$ To reiterate, we are required to separate $y_2$ such that the $\tau_4^{-1}$ does not intersect the first term. This obstruction suggests the first $y_2$ part is $\tau_2^{-1}$ and the second term is the remaining $\rho_2^{-1}C_{12}\rho_2^{-1}$. This construction yields the desired map $$\rho_2^{-1}C_{12}\rho_2^{-1}\tau_4^{-1}\rho_3^{-2}\tau_2^{-1}\rho_4^{-1}\rho_1^{-1}.$$
\begin{table}[h]
  \caption{Toroidal Example 2 Intersections} \label{Ex3_Ints}
\centering
  \begin{tabular}{ | l | c |r| }
    \hline
    Points & Time &Crossing \\ \hline
    $(y_2,y_4)$ & $3\epsilon$& $y_2$ over $y_4$ \\ \hline
    $(y_2, y_3)$ & $\frac{1}{2}$&$y_3$ over $y_2$\\ \hline
    $(y_1,y_4)$& $1-5\epsilon$& $y_4$ over $y_1$\\\hline
    $(y_1,y_2)$ & $1-3\epsilon$& $y_2$ over $y_1$\\\hline
    $(y_3,y_4)$ & $1-\epsilon$& $y_4$ over $y_3$\\ \hline
  \end{tabular}
  \end{table}
\end{Ex(t)}
\subsection*{Thurston Ball Revisited}
Recall the toroidal examples in Table 1 represent fibrations with fiber a four-times punctured torus. Each such fibration corresponds to an element $\psi\in H^1(T^3, \mathbb{Z})$ on the Thurston ball. We claim $\psi$ must lie in the open cone over a quadrilateral face. This claim follows from Theorem 2.3 in \cite{McT} and the fact that the fiber has exactly four punctures. Indeed, writing $\psi = (a,b,c)$, we require 
\begin{align*}
\psi(L_1)&= a = \pm 1,\\
\psi(L_2)&= b = \pm 1,\\
\psi(L_3)& = c = \pm 1,\\
\psi(L_4)& = a+b+c = \pm 1.
\end{align*}
These equations yield six possibilities for $\psi$: $(-1,1,1), (1,-1,1), (1,1,-1), (-1,-1,1),
 (-1,1,-1)$, $(1,-1,-1)$. Referring to a diagram of the Thurston ball, one can easily verify each of these lives on a quadrilateral face.
%%%%%%%%%%%%%%%%%%%%%%%%%%%%%%%%%%%%%%%%%%%%%%%%%%5%%%%%%%%%%%%%%%%%%%%%%%%%%%%%%%%%%%%%%%%%%%%
%%%%%%%%%%%%%%%%%%%%%%%%%%%%%%%%%%%%%%%%%%%%GRAPH LINKS%%%%%%%%%%%%%%%%%%%%%%%%%%%%%%%%%%%%%%%%%%%%
%%%%%%%%%%%%%%%%%%%%%%%%%%%%%%%%%%%%%%%%%%%%%%%%%%%%%%%%%%%%%%%%%%%%%%%%%%%%%%%%%%%%%%%%%%%%%%
\section{4-Manifolds from Graph Links}\label{sec-GL}
Here, we study a family of two-component links with different fibrations over a circle, resulting in inequivalent symplectic structures after taking the  product with a circle.  We start by describing the open 3-manifolds.  Let $M^{(2n)} = S^3\backslash K^{(2n)}$, where $K^{(2n)}$ is the graph link pictured in Figure \ref{graph_link} below. The details of this diagram are given in \cite{Graph_Link}, where the third author showed the existence of $n+1$ inequivalent symplectic structures coming from different fibrations of $M^{(2n)}$.
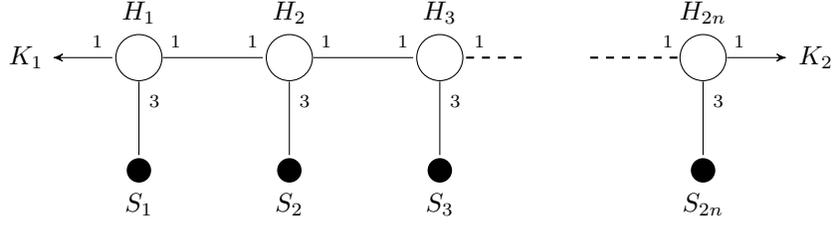
\begin{figure}[h!]
\centering
\begin{tikzpicture}[>=stealth',shorten >=1pt,auto,node distance=2cm]
\node[] (k1) at (-1.5,0){\small $K_1$};
\node[](k2) at (9,0){\small $K_2$};
\node[label={\small$H_1$},scale=1.7,circle, draw](v1) at (0,0) {};
\node[label={\small$H_2$}, scale=1.7, circle, draw](v2) at (2,0) {};
\node[label={\small$H_3$}, scale=1.7, circle, draw](v3) at (4,0) {};
\node[label={\small$H_{2n}$}, scale=1.7, circle, draw](vn) at (7.5,0){};
\node[label=below: {\small$S_1$}, circle, fill=black, scale=0.9](s1) at (0,-1.5){};
\node[label=below: {\small$S_2$}, circle, fill=black, scale=0.9](s2) at (2,-1.5) {};
\node[label=below: {\small$S_3$}, circle, fill=black, scale=0.9](s3) at (4,-1.5) {};
\node[label=below: {\small$S_{2n}$}, circle, fill=black, scale=0.9](sn) at (7.5,-1.5) {};
 \path[<-](k1) edge node[above, xshift=5pt] {\tiny $1$} (v1);
\draw[-] (v1)  to node[very near start]{\tiny $1$} node[very near end]{\tiny $1$}  (v2);
\draw[-] (v2)  to node[very near start]{\tiny $1$} node[very near end]{\tiny $1$}  (v3);
\path[-](v1) edge node[right, yshift=7pt] {\tiny $3$}(s1);
\path[-](v2) edge node[right, yshift=7pt] {\tiny $3$}(s2);
\path[-](v3) edge node[right, yshift=7pt] {\tiny $3$}(s3);
\path[-](vn) edge node[right, yshift=7pt] {\tiny $3$}(sn);
 \draw[dashed, thick] (4.35,0)--(5.2,0) node[draw=none,fill=none,midway,above, xshift=-7pt] {\tiny $1$};
 \draw[dashed, thick] (6,0)--(7.22,0) node[draw=none,fill=none,midway,above, xshift=12pt] {\tiny $1$};
 \draw[->] (vn) edge node[above, xshift=-7pt]{\tiny $1$} (k2);
\end{tikzpicture}
 \caption{Diagram of $K^{(2n)}$}
 \label{graph_link}
\end{figure}
A fibration of $M^{(2n)}$ is given by a choice of $(m_1, m_2)\in H^1(S^3\backslash K^{(2n)}), \mathbb{Z})\cong \mathbb{Z}^2$ satisfying the equations $$3^im_1 + 3^{2n-i+1}m_2\neq 0, \text{\ for  all  } 1\leq i\leq 2n.$$ Details for such a fibration (and graph link theory in general) are worked out by Eisenbud and Neumann in \cite{Eisenbud}. In particular, let $h$ denote the monodromy and $h_*$ the induced map on the homology of the fiber. \cite[Theorem 13.6]{Eisenbud} shows that there is an integer $q$ such that $(h_*^q-1)^2 = 0$. Thus, the Jordan decomposition of $h_*$ only has blocks of size 1 or 2. Furthermore, with the same $q$, \cite{Eisenbud} computed the characteristic polynomial of $h_*|_{\text{Im}(h^q_*-1)}$, denoted $\Delta'(t)$. It turns out that the roots of $\Delta'(t)$ correspond to the eigenvalues of $h_*$ with size 2 Jordan blocks. Moreover, the multiplicity of each root $\lambda_i$ in $\Delta'(t)$ gives the number of size 2 blocks for $\lambda_i$. 
\par We first introduce some notation which will be used in the definition of $\Delta'(t)$. Fix a fibration $(m_1, m_2)$. Let $\mathcal{E} = \{E_1, \cdots, E_{2n-1}\}$ be the set of edges connecting the white nodes in Figure \ref{graph_link}. Specifically, edge $E_i$ connects nodes labeled $H_i$ and $H_{i+1}$. For each $E_i \in \mathcal{E}$, we define an integer $d_{E_i}$ as follows. Take the path in $K^{(2n)}$ from the arrowhead of $K_1$ to halfway through edge $E_i$ (passing through nodes $H_1,H_2,\cdots, H_i$). Let $\ell_{E_i,1}$ denote the product of all weights on edges not contained in the path but are adjacent to vertices in the path. Similarly, we can take the path from the arrowhead of $K_2$ to halfway through edge $E_i$ and define $\ell_{E_i,2}$ analogously. Set $$d_{E_i} = \gcd(m_1\ell_{E_i,1}, m_2\ell_{E_i,2}).$$ Using Figure \ref{graph_link} as reference, we can easily compute that $\ell_{E_i,1} = 3^{i}$ and $\ell_{E_i,2} = 3^{2n-i}$. This simplifies the formula for $d_E$ to 
\begin{equation}\label{$d_E$}
d_{E_i}=\gcd(3^{i}m_1, 3^{2n-i}m_2).
\end{equation}
For each vertex $H_i$, we define an integer $d_{V_i}$ by the formula
\begin{equation}\label{$d_V$}
d_{V_i}=\left\{
 \begin{array}{cc}
 \gcd(d_{E_{i-1}}, d_{E_i}), & 1< i <2n\\
 \gcd(m_1, d_{E_1}),& i= 1\\
 \gcd(m_2, d_{E_{2n-1}}),& i= 2n
 \end{array}\right.
 \end{equation}
 With these definitions in place, the (restricted) characteristic polynomial takes the form
\begin{equation*}
\Delta'(t) = (t^d-1)\prod_{i=1}^{2n-1}(t^{d_{E_i}}-1)/\prod_{i=1}^{2n} (t^{d_{V_i}}-1),
\end{equation*}
where $d=\gcd(m_1,m_2)$. To understand the details of the theorem to follow, we will use the $n=2$ case, i.e. $K^{(4)}$, as an illustrative example. Figure \ref{graph_link2} demonstrates how $d_{E_1}=\gcd(3m_1, 3^3m_2)$ is calculated.
\begin{figure}[h!]
\centering
\begin{tikzpicture}[>=stealth',shorten >=1pt,auto,node distance=2cm]
\node[] (k1) at (-1.5,0){$K_1$};
\node[](k2) at (7.5,0){$K_2$};
\node[label=\small$H_1$,scale=1.7,circle, draw](v1) at (0,0) {};
\node[label=\small$H_2$, scale=1.7, circle, draw](v3) at (2,0) {};
\node[label=\small$H_3$, scale=1.7, circle, draw](v5) at (4,0) {};
\node[label=\small$H_{4}$, scale=1.7, circle, draw](v7) at (6,0){};
\node[label=below: \small$S_1$, circle, fill=black, scale=0.9](s2) at (0,-1.5){};
\node[label=below: \small$S_2$, circle, fill=black, scale=0.9](s4) at (2,-1.5) {};
\node[label=below: \small$S_3$, circle, fill=black, scale=0.9](s6) at (4,-1.5) {};
\node[label=below: \small$S_{4}$, circle, fill=black, scale=0.9](s8) at (6,-1.5) {};
 \path[<-](k1) edge (v1);
  \path[->](v7) edge (k2);
 %\path[-] (v1)  edge [above] node {} (v3);
 \path[-] (v3)  edge [above] node {} (v5);
 \path[-] (v5)  edge [above] node {} (v7);
\path[-](v1) edge node[right, yshift=7pt] {\tiny $3$}(s2);
\path[-](v3) edge node[right, yshift=7pt] {\tiny $3$}(s4);
\path[-](v5) edge node[right, yshift=7pt] {\tiny $3$}(s6);
\path[-](v7) edge node[right, yshift=7pt] {\tiny $3$}(s8);
\node[] (e) at (0.9,0){};
\node[] (f) at (1.1,0){};
\path[-, line width=0.1cm] (e)  edge [above] node {} (v1); 
\path[-, line width=0.1cm] (v1)  edge [above] node {} (k1); 
\path[-, line width=0.1cm] (f)  edge [above] node {} (v3); 
\path[-, line width=0.1cm] (v3)  edge [above] node {} (v5); 
\path[-, line width=0.1cm] (v5)  edge [above] node {} (v7);
\path[-, line width=0.1cm] (v7)  edge [above] node {} (k2);
\end{tikzpicture}
 \caption{Paths $\ell_{E_1,1}$ and $\ell_{E_1,2}$ of $d_{E_1}$}
 \label{graph_link2}
\end{figure}
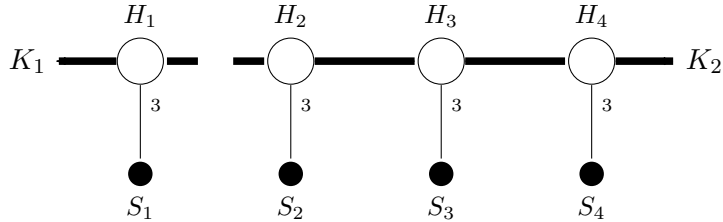

Now, define $X^{(n)}=S^1\times M^{(2n)}$ and let $\deg\Delta'(t)$ denote the degree of the restricted characteristic polynomial $\Delta'(t)$. Since $\deg\Delta'(t)$ is the number of Jordan blocks of size 2, which equals the number of blocks of size {\it at least} 2, it follows
\begin{align*}
p_2^+&=b_2(X^{(n)})+1+\deg\Delta'(t),\\
p_2^-&=b_2(X^{(n)})+\deg\Delta'(t).
\end{align*}
In the case of a fibration represented by coprime $(m_1, m_2)$, there are two possibilities: 3 divides exactly one of $m_1$ or $m_2$,  or 3 neither divides $m_1$ nor $m_2$. It turns out $p_2^+$ can distinguish these two possibilities and in the first case provides information about the power of $3$ dividing $m_1$ or $m_2$. We give the exact statement below.
\begin{Th}\label{GL-thm}
Let $(m_1,m_2)$ be coprime, representing a fibration of $M^{(2n)}$. By reversing the roles of $m_1$ and $m_2$ if necessary, we write $m_1= 3^kq$ with $\gcd(q,3)=1$ and assume $\gcd(3,m_2)=1$. It follows that
\[
p_2^+=\left\{ \begin{array}{lc}b_2(X^{(n)})+ 3^{n+k-\lceil \frac{k}{2}\rceil}-3^k, &k\leq 2n-4\\b_2(X^{(n)})+2\cdot 3^k, &2n-3\leq k\leq 2n-2\\b_2(X^{(n)})+1,& k\geq 2n-1 \end{array}\right.\]
\end{Th}
\begin{proof}
For convenience, we denote by $r_k$ the quantity $\lfloor\frac{2n-k}{2}\rfloor$. Note that $r_k$ is chosen such that $\min(k+i, 2n-i) = k+i$ for $i\leq r_k$. With this fact in mind, we first expand the expressions for $d_{E_i}$:
$$ d_{E_i} = \gcd(3^{k+i}q, 3^{2n-i}m_2) = 3^{\min(k+i, 2n-i)}$$
A simple consideration shows
\[ d_{E_i} = \begin{cases} 
      3^{k+i}, & i \leq r_k \\
      3^{2n-i}, & i > r_k
   \end{cases}
\]

Next, we expand the expressions for $V_i$. For $1 < i < 2n$, we have 
\[d_{V_i} =\gcd(d_{E_{i-1}}, d_{E_i}) = \begin{cases} 
      d_{E_{i-1}}, & i \leq r_k \\
      d_{E_i}, & i > r_k
   \end{cases}
\]
Verifying the cases where $i \leq r_k$ or $i-1 > r_k$ is straightforward. For the case $i=r_k+1$, we have $d_{V_i} = 3^{\min(k+r_k, 2n-r_k-1)}$. A case by case analysis on the parity of $k$ reveals
\begin{align*}
&r_k \geq n - \frac{k}{2} - \frac{1}{2}\\
& \implies r_k + r_k \geq 2n - k - 1\\
& \implies r_k + k \geq 2n - r_k - 1
\end{align*}
and so $d_{V_i} = 3^{2n-r_k-1} = d_{E_{r_{k+1}}} = d_{E_i}$ as required. It remains to compute $d_{V_1}$ and $d_{V_{2n}}$.
\begin{align*}
d_{V_1}&= \gcd(m_1, d_{E_1})=\gcd(3^kq, 3^{\min(k+1, 2n-1)})=3^{\min(k, 2n-1)},\\
d_{V_{2n}}&= \gcd(m_2, d_{E_{2n-1}})  = \gcd(m_2, 3^{\min(k+2n-1, 1)})=1,
\end{align*}
where $d_{V_{2n}}$ follows from the fact that $\gcd(m_2, 3) = 1$.
With these definitions, we turn to the formula for $\Delta'(t)$. Its expansion will depend on whether or not there is some index $1<i<2n$ for which $i = r_k$. In particular, such an index exists only if $r_k >1$.
\begin{Case} ($r_k>1$)
We have 
\begin{align*}
\Delta'(t)&=\frac{(t-1)(t^{d_{E_1}}-1)\cdots(t^{d_{E_{r_k}}}-1)\cdots (t^{d_{E_{2n-1}}}-1)}{(t^{d_{V_1}}-1)\cdots (t^{d_{V_{r_k}}}-1)\cdots (t^{d_{V_{2n}}}-1)}\\
&=\frac{(t^{d_{E_{r_k}}}-1)}{(t^{d_{V_1}}-1)} = \frac{(t^{3^{k+r_k}}-1)}{(t^{3^{\min(k, 2n-1)}}-1)} \\
&=\frac{(t^{3^{k+r_k}}-1)}{(t^{3^k}-1)},
\end{align*}
since for $r_k \geq 1$ we have $k < k+1 < 2n-1$.
\end{Case}
\begin{Case} ($r_k\leq 1$)
Now we have $d_{E_i} = d_{V_i}$ for all $1<i<2n$. Furthermore,
\[ d_{E_1} = \begin{cases} 
      3^{k+1}, & r_k = 1 \\
      3^{2n-1}, & r_k < 1.
   \end{cases}
\]
Recall from the definition of $r_k$, we also have
\[ \min(k, 2n-1) = \begin{cases} 
      k, & r_k = 1 \\
      2n-1, & r_k < 1.
   \end{cases}
\]
Combining the above equalities yields,   
\[ \Delta'(t)=\frac{t^{d_{E_1}}-1}{t^{d_{V_1}}-1} = \frac{t^{d_{E_1}}-1}{t^{3^{\min(k, 2n-1)}}-1}=\begin{cases} 
      \frac{t^{3^{k+1}}-1}{t^{3^k}-1}, & r_k = 1 \\
      1, & r_k < 1.
   \end{cases}
\]
\end{Case}
To acquire a more straightforward understanding of the three cases $r_k <1, r_k = 1$, and $r_k >1$ let us rewrite $r_k = \lfloor \frac{2n-k}{2}\rfloor = n - \lceil \frac{k}{2}\rceil$. Then $r_k = 1$ implies $\lceil \frac{k}{2}\rceil=n-1$ and so $k = 2n-2$ or $2n-3$. Finally, we conclude 
\begin{align}\label{degDp}
 \deg \Delta'(t) = \begin{cases} 
      3^{n+k-\lceil \frac{k}{2}\rceil} - 3^k, & k \leq 2n-4 \\
      2\cdot 3^k, & 2n-3\leq k\leq 2n-2\\
      0, & k\geq 2n-1,
   \end{cases}
\end{align}
from which the result follows.
\end{proof}

To illustrate some of the details in the proof of Theorem \ref{GL-thm}, we calculate the $K^{(4)}$ case explicitly. Here, we have $r_k = 2 - \lceil\frac{k}{2}\rceil$ and so $r_0 = 2$, $r_1=r_2=1$ and $r_i <1$ for $i\geq 3$. Note that $r_k \leq 1$ for all $k >0$. In particular, for $k >0$ and $i>1$ we should expect $d_{E_i} = 3^{4-i}=d_{V_i}$. To confirm, we compute
\begin{align*}
d_{E_2}&=\gcd(3^{k+2}q,3^2)=3^2,\\
d_{E_3}&=\gcd(3^{3+k}q,3)=3,\\
d_{V_2}&=\gcd(\min(3^{k+1},3^3),3^2)=\min(3^{k+1},3^2)=3^2,\\
d_{V_3}&=\gcd(3^2,3)=3,\\
d_{V_4}&=\gcd(m_2,3)=1.
\end{align*}
For $i=1$, we have 
\begin{align*}
d_{E_1}&=\gcd(3^{k+1}q,3^3m_2) =\min(3^{k+1}, 3^3),\\
d_{V_1}&=\gcd(m_1, d_{E_1}) =\gcd(3^kq, 3^{\min(3, k+1)})={\min(3^3, 3^k)}.
\end{align*} 
We can now explicitly compute $\Delta'(t)$:
\begin{align*}
\Delta'(t)&=\frac{(t-1)(t^{\min(3^{k+1},3^3)}-1)(t^9-1)(t^3-1)}{(t^{\min(3^k,3^3)}-1)(t^{9}-1)(t^3-1)(t-1)}\\
&=\frac{t^{3^2\min(3^{k-1},3)}-1}{t^{3\min(3^{k-1},3^2)}-1}\\
&= \left\{ \begin{array}{cc}t^6+t^3+1,&k=1\\t^{18}+t^9+1,&k=2\\1,& \;k\geq 3. \end{array}\right.
\end{align*}
Indeed, this calculation agrees with the general formula \eqref{degDp}
\begin{align*}
\deg(\Delta'(t)) &= \left\{ \begin{array}{cc}6 = 2\cdot3,&k=1\\18=2\cdot3^2,&k=2\\0,& \;k\geq 3.
\end{array}\right.
\end{align*}

We will also use this $K^{(4)}$ example to demonstrate another interesting property of the $X^{(2n)}$ family of open symplectic 4-manifolds.  It can be seen in this $n=2$ case that every possible value of $p_2^+$, in Theorem \ref{GL-thm}, may be achieved for fibrations coming from any of the faces on the Thurston ball. In Figure \ref{K4ball} below, we give the diagram of the Thurston ball for $M^{(4)}$ taken from \cite[Figure 3]{Graph_Link}.
 \begin{figure}[H]
 \includegraphics[scale=0.75]{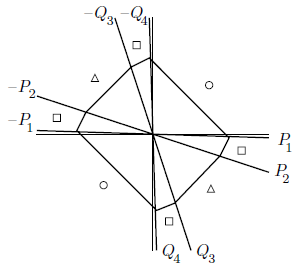}
 \caption{Thurston ball for $M^{(4)}$}\label{K4ball}
 \end{figure}
 The eight faces come in three types, denoted by the circular, square, and triangular symbols. The open cone over each face is determined by the rays passing through the primitive elements 
 \begin{align*}
 P_i&= (3^{5-2i}, -1),\\
 Q_i&= (1, -3^{2i-5}).
 \end{align*}
Consider, for example, the open cone over the circular face in the first quadrant, as the consideration for the other two types of faces are similar. This region is defined by the elements $\{\left.(x,y) \right| -\frac{1}{27}x < y < 27x,\, 0 <x <\infty\}$, laying between the $P_1$ and $-Q_4$ rays. Furthermore, it contains the points $\phi_i = (3^i, 1)$ for $0 \leq i \leq 3$. By Theorem \ref{GL-thm}, the corresponding symplectic 4-manifolds $(X^{(4)}, \omega_i)$ have the following primitive dimension: 
 \begin{align*}
 p_2^+(X^{(4)}, \omega_1)&= b_2 +9,\\
  p_2^+(X^{(4)}, \omega_2)&= b_2 +7,\\
   p_2^+(X^{(4)}, \omega_3)&= b_2 +19,\\
    p_2^+(X^{(4)}, \omega_4)&= b_2 +1.
 \end{align*}  
Moreover, by symmetry of the Thurston ball, the fibrations $-\phi_i$ lie on the cone over the other circular face and also satisfy the above equations. Hence, the pair of circular faces have fibrations with all possible dimensions of $p_2^+$ given in Theorem \ref{GL-thm}. A similar analysis shows the same property holds true for the square and triangular faces as well.

Since fibrations within a single face of the Thurston ball are all equivalent, i.e. there is a smooth isotopy on the one-form $d\pi$ linking any two fibrations, the above provides examples of symplectic 4-manifolds from equivalent 3-manifold fibrations having different primitive cohomology.  
 
%%%%%%%%%%%%%%%%%%%%%%%%%%%%%%%%%%%%%%%%%%%%%%%%%%5%%%%%%%%%%%%%%%%%%%%%%%%%%%%%%%%%%%%%%%%%%%%
%%%%%%%%%%%%%%%%%%%%%%%%%%%%%%%%%%%%%%%%%%%%DISCUSSION%%%%%%%%%%%%%%%%%%%%%%%%%%%%%%%%%%%%%%%%%%%%
%%%%%%%%%%%%%%%%%%%%%%%%%%%%%%%%%%%%%%%%%%%%%%%%%%%%%%%%%%%%%%%%%%%%%%%%%%%%%%%%%%%%%%%%%%%%%%
\section{Discussion}\label{sec-D}
We conclude with some remarks. Recall from Section \ref{examples} that the toroidal and spherical fibrations came from inequivalent faces of the Thurston ball. The cohomological value $p_2^+$ takes different values for these two types of fibration. The ability to distinguish these two distinct faces in this McMullen-Taubes example arises from the fact that one monodromy is Torelli, while the other is not. In fact, $p_2^+$ will {\it always} differ for a Torelli and non-Torelli fibration. 

\par On the other hand, we have also seen from the graph link family of examples that the value of $p_2^+$ can differ quite a bit for fibrations coming from the {\it same} face of the Thurston ball.  For sure, $p_2^+$ can give more refined information on each face of the Thurston ball; within a face,  $p_2^+$ can take on different values for various fibrations of $X$ whereas (the classical Betti number) $b_2$ does not.

\par  It is still however an open question as to how this distinguishing value $p_2^+$ can be used with other invariants to probe more deeply the relation between different fibrations and their associated symplectic structures. As mentioned, starting with a surface $\Sigma$ and a monodromy $f$, we can construct a $Y_f$ and an associated symplectic manifold $(X, \omega_f)$. The dimension $p_2^+$ ultimately comes down to the algebraic structure of the monodromy $f$ on the surface. But as the graph links examples showed, the value of $p_2^+$ alone is not sufficient in determining which face the monodromy corresponds to. We do expect that the other invariants associated with the underlying $A_3$-algebra on $PH^*(X, \omega_f)$ \cite{TsengIII} will be able to provide even more refined information. The $A_3$-algebra on the primitive differential forms not only induce a product structure on $PH^*(X, \omega_f)$, but it is also possible to define Massey products on $PH^*(X, \omega_f)$ as well. And as a specific example of how they may provide more information, the triple Massey product can be shown to count the number of Jordan blocks of size exactly three. (See \cite{Gibson} for more information.)  It would be interesting to build up a dictionary relating the primitive cohomological invariants of the symplectic 4-manifold with structures on the fibered 3-manifold, and that of the monodromy $f$, an element of the mapping class group. 
%%%%%%%%%%%%%%%%%%%%%%%%%%%%%%%%%%%%%%%%%%%%%%%%%%%%%%%%%%%%%%%%%%%%%%%%%%%%%%%%%%%%%%%%%%%%%%
%%%%%%%%%%%%%%%%%%%%%%%%%%%%%%%%%%%%%%%%%%%%%%APPENDIX%%%%%%%%%%%%%%%%%%%%%%%%%%%%%%%%%%%%%%%%%%%%
%%%%%%%%%%%%%%%%%%%%%%%%%%%%%%%%%%%%%%%%%%%%%%%%%%%%%%%%%%%%%%%%%%%%%%%%%%%%%%%%%%%%%%%%%%%%%%
\section*{Appendix}
Here we provide the details of setting up the fibration structure and converting monodromies appropriately so that they can be entered into SnapPy.
Let $\mathbb{T}^3$ denote the $3$-torus. We view it as the cube $[0,1]^3$ under the identification $(x,y,z) \sim (x+p, y+q,z+r)$ for integers $p,q,r$. The axes $i,j,k$ and their sum $i+j+k$ form four lines in the cube $L_1, L_2, L_3, L_4$, respectively. By choosing different bases $(v_1, v_2, v_3)$ for the cube and displacing the four lines we may fiber $\mathbb{T}^3-\{L_1,L_2,L_3,L_4\}$ in different ways as follows. First we shift the four lines from the origin by 
\begin{align*}
L_1&=(x, -\epsilon, 3\epsilon),\\
L_2&=(\epsilon,y,-3\epsilon),\\
L_3&=(-\epsilon, \epsilon, z),\\
L_4&= (x=y=z).
\end{align*}
\par Next we choose a basis $v_1 = (a_1, a_2, a_3)$, $v_2=(1,1,0)$, $v_3=(0,1,1)$. Initially $v_1$ may be any vector which gives a non-zero determinant, specifically, $a_1-a_2+a_3\neq 0$. For brevity, let us denote $A:= det(v_1,v_2,v_3) = a_1-a_2+a_3$. Choosing to fiber along $v_1$, each fiber has the form $\Sigma_t = tv_1 + \alpha v_2 + \beta v_3$ for $t\in [0,1]$. $\Sigma_t$ is $\mathbb{T}^2$ with four punctures denoted $x_1(t), x_2(t), x_3(t), x_4(t)$ coming from the respective lines $L_i$. To verify that each line $L_i$ intersects the fiber exactly once we must solve the following system of equations:
\begin{align*}
L_1:& \begin{pmatrix}1& 1\\0&1 \end{pmatrix}\begin{pmatrix}\alpha\\ \beta \end{pmatrix}= \begin{pmatrix} -\epsilon-ta_2\\ 3\epsilon-ta_3\end{pmatrix}\\
L_2:& \begin{pmatrix}1& 0\\0&1 \end{pmatrix}\begin{pmatrix}\alpha\\ \beta \end{pmatrix}= \begin{pmatrix} \epsilon-ta_1\\ -3\epsilon-ta_3\end{pmatrix}\\
L_3:& \begin{pmatrix}1& 0\\1&1 \end{pmatrix}\begin{pmatrix}\alpha\\ \beta \end{pmatrix}= \begin{pmatrix} -\epsilon-ta_1\\ \epsilon-ta_2\end{pmatrix}\\
L_4:& \begin{pmatrix}0& -1\\1&-1 \end{pmatrix}\begin{pmatrix}\alpha\\ \beta \end{pmatrix}= \begin{pmatrix} t(a_2-a_1)\\ t(a_3-a_1)\end{pmatrix}
\end{align*}
Solving these systems for the $(\alpha, \beta)$ coordinates of the marked points $x_i(t)$ yields
\begin{align*}
x_1(t)&=(-4\epsilon, 3\epsilon) + (a_3-a_2,-a_3)t,\\
x_2(t)&=(\epsilon,-3\epsilon)+(-a_1,-a_3)t,\\
x_3(t)&=(-\epsilon, 2\epsilon) + (-a_1, a_1-a_2)t,\\
x_4(t)&=(0,0)+(a_3-a_2,a_1-a_2)t.
\end{align*}
To align with the notation of \cite{Birman}, we relabel the points with respect to their first coordinate position, in increasing order, as $y_1(t) = x_1(t)$, $y_2(t) = x_3(t)$, $y_3(t)=x_4(t)$, $y_4(t)=x_2(t)$. Under this new setting, the formulas for the points become
\begin{align*}
y_1(t)&=(-4\epsilon, 3\epsilon) + (a_3-a_2,-a_3)t,\\
y_2(t)&=(-\epsilon, 2\epsilon) + (-a_1, a_1-a_2)t,\\
y_3(t)&=(0,0)+(a_3-a_2,a_1-a_2)t,\\
y_4(t)&=(\epsilon,-3\epsilon)+(-a_1,-a_3)t.\\
\end{align*}
\par Next, we verify that none of the $y_i(t)$ intersect for any value of $t$. Notice $y_2$ and $y_3$ have the same second component in the $t$ variable but differ by the $\epsilon$-term constant so they will never intersect. We can apply a similar argument to the pairs $(y_1,y_3),(y_1,y_4)$, and $(y_2,y_4)$. Lastly, by considering the (separate) systems of equations $y_1(t)=y_2(t)$ and $y_3(t)=y_4(t)$, one can easily see no solution exists.
\par Let $\Sigma_{1,4}$ be the $2-$torus with four punctures and $\text{Mod}(\Sigma_{1,4})$ its mapping class group (which fixes the punctures setwise). Furthermore, let $\mathcal{P}\text{Mod}(\Sigma_{1,4})$ denote the {\it pure} mapping class group, the set of mapping class elements fixing the punctures pointwise. We set 
 \begin{equation}\label{generators}
 H_1(\Sigma) = \langle a_0, a_1, a_2, a_3, b_0\rangle,
 \end{equation}
 where $a_i$ is the homology curve between punctures $i$ and $i+1$ for $i>0$ and $a_0$ is between marked point 1 and 4. $b_0$ is the homology longitudinal curve, not enclosing any punctures. These curves have algebraic intersection numbers $a_i\cdot a_j= 0$ for $i\neq j$ and $a_i \cdot b_0 = 1$. \cite{Birman} introduces the following elements (pictured below) and shows the Dehn twists along them that generate the pure mapping class group. In our setting we have $\mathcal{P}\text{Mod}(\Sigma_{1,4}) = \langle \mathcal{P}ush(\rho_i), \mathcal{P}ush(\tau_i) \rangle$, $1\leq i \leq 4$. Here, $\mathcal{P}ush(\gamma)$ is the point pushing map along $\gamma$. We also summarize some of the important relations to be used later:
 \begin{align*}
 &[\tau_i, \tau_j] = [\rho_i, \rho_j] = 1,\\
 &A_{ij} = \rho_i\tau_j^{-1}\rho_i^{-1}\tau_j, \quad C_{ij} = \tau_i\rho_j^{-1}\tau_i^{-1}\rho_j,\\
&\text{for\ } 1\leq i<j<k \leq 4.
 \end{align*}
 \begin{figure}[h!]\label{gens}
 \includegraphics[scale=0.5]{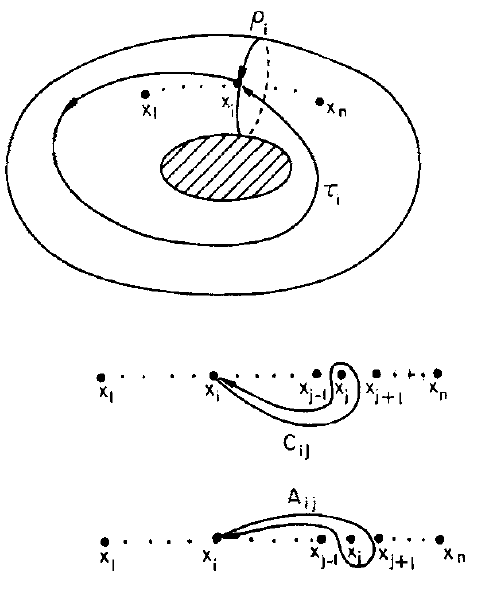}
 \caption{Diagram of generators taken from \cite{Birman}}
 \end{figure}
For a more in depth discussion and outline of a proof for these identities, see \cite{Birman}. We note that the formulas here differ slightly from \cite{Birman} as our choice of orientation is not the same. Moreover, we use functional composition, (right to left) as opposed to algebraic. In order to use SnapPy \cite{SnapPy}, we need to express $\mathcal{P}ush(\rho_i)$ and  $\mathcal{P}ush(\tau_i)$ in terms of Dehn twists along the curves in (\ref{generators}). The trick is to use the following fact (4.7 proven in \cite{Farb}), which states
 \begin{fact*}
 Let $\alpha$ be a simple loop in a surface $S$ representing an element of $\pi_1(S, x)$, Then 
 $\mathcal{P}ush([\alpha]) = T_aT_b^{-1},$
 where $a$ and $b$ are isotopy classes of the simple closed curves in $S-x$ obtained by pushing $\alpha$ off itself to the left and right, respectively.
 \end{fact*}
 That is, we take an annular neighborhood of $\alpha$ bounded by curves $a$ and $b$ and then take the product of their Dehn and inverse Dehn twists, respectively. From this construction, we can immediately obtain that 
\begin{equation}
\mathcal{P}ush(\rho_i) = T_{a_{i-1}}T_{a_i}^{-1}.
 \end{equation}
 \par For the $\tau_i$ curves, we need to find an annular boundary to work with. We introduce the longitudinal homology curves $b_i$, which enclose the punctures $1,2, \cdots, i$ [``over'' $1,2,\cdots, i$ and ``under'' $i+1,\cdots, 4$]. Thus $b_0$ agrees with the previous homology generator introduced, $b_1$ passes over puncture 1 and misses 2,3,4, and so on. The point of introducing these curves is that now $\tau_i$ has an annular neighborhood bounded by $b_{i-1}$ and $b_i$. By consulting the diagrams to determine proper orientation it follows that 
 \begin{equation}\label{tauPush}
\mathcal{P}ush(\tau_i) = T_{b_i}T_{b_{i-1}}^{-1}.
 \end{equation}
 \par Next, we need to convert Equation \ref{tauPush} into Dehn twists only involving the homology generators given in \ref{generators}. First, we observe that we may express $[b_i] = [a_0]+[b_0]-[a_i]$, which can be verified by constructing the fundamental square for the torus with the relevant curves. An example diagram in Figure \ref{Example} is given for the $[b_1]$ case.
 \begin{figure}[h!]
 \centering
 \begin{tikzpicture}
 \draw (0,0) grid +(3,3);
\fill (canvas cs: x=0.5cm, y =2.5cm) circle (2pt);
\fill (canvas cs: x=1.5cm, y =1.5cm) circle (2pt);
\fill (canvas cs: x=2.5cm, y =0.5cm) circle (2pt);
\draw[->>, line width=0.5mm, red] (3,2)--(3,3);
\draw[->>, line width=0.5mm, red] (3,3)--(1,3);
\draw[->>,line width=0.5mm, red] (1,3)--(1,0);
\draw[->>, line width=0.5mm, red] (1,0)--(0,0);
\draw[->>, line width=0.5mm, red] (0,0)--(0,2);
\draw[->>, line width=0.5mm, blue] (3,2)--(0,2);
\node[] at (1.5,2.2) {$b_1$};
\node[] at (-0.2,1.2) {$a_0$};
\node[] at (2,3.2) {$b_0$};
\node[] at (0.7,1.5) {$-a_1$};
 \end{tikzpicture}
 \caption{Diagram for $b_1$ Expression}\label{Example}
 \end{figure}
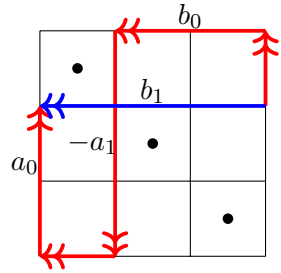
  One can straightforwardly check that $T_{a_i}T_{b_0}([a_0]) = [a_0]+[b_0]-[a_i] = [b_i]$. Fact 3.7 in \cite{Farb} states $T_{f(a)} = fT_af^{-1}$, which we can apply to our situation by setting $a= a_0$ and $f=T_{a_i}T_{b_0}$. This fact then yields 
\begin{equation}\label{sub}
 T_{b_i} = T_{a_i}T_{b_0}T_{a_0}T_{b_0}^{-1}T_{a_i}^{-1}.
 \end{equation}
 Finally, substituting formula \ref{sub} into equation \ref{tauPush} leads to our desired expression
 \begin{equation}\label{newTauPush}
 \mathcal{P}ush(\tau_i)=T_{a_{i-1}}T_{b_0}T_{a_0}^{-1}T_{b_0}^{-1}T_{a_{i-1}}^{-1}T_{a_i}T_{b_0}T_{a_0}T_{b_0}^{-1}T_{a_i}^{-1}.
 \end{equation}
 %
%\newpage
%\clearpage

%


\bigskip\bigskip\bigskip

\begin{bibdiv}
\begin{biblist}
\bib{Birman}{article}{
   author={Birman, J. S.},
   title={On braid groups},
   journal={Comm. Pure Appl. Math.},
   volume={22},
   date={1969},
   pages={41--72},
   issn={0010-3640},
   review={\MR{0234447}},
   doi={10.1002/cpa.3160220104},
}
%
\bib{SnapPy}{misc}{
     author={Culler, M.},
     author={Dunfield, N. M.},
     author={Goerner, M.},
     author={Weeks, J. R.},
     title={Snap{P}y, a computer program for studying the geometry and topology of $3$-manifolds},
     note={Available at \url{http://snappy.computop.org}},
}
%
\bib{Eisenbud}{book}{
   author={Eisenbud, D.},
   author={Neumann, W.},
   title={Three-dimensional link theory and invariants of plane curve
   singularities},
   series={Annals of Mathematics Studies},
   volume={110},
   publisher={Princeton University Press, Princeton, NJ},
   date={1985},
   pages={vii+173},
   isbn={0-691-08380-0},
   isbn={0-691-08381-9},
   review={\MR{817982}},
}
%
\bib{Farb}{book}{
   author={Farb, B.},
   author={Margalit, D.},
   title={A primer on mapping class groups},
   series={Princeton Mathematical Series},
   volume={49},
   publisher={Princeton University Press, Princeton, NJ},
   date={2012},
   pages={xiv+472},
   isbn={978-0-691-14794-9},
   review={\MR{2850125}},
}
%
\bib{Gibson}{book}{
   author={Gibson, M.},
   title={Properties of the $A_\infty$-Structure on Primitive Forms and Its
   Cohomology},
   note={Thesis (Ph.D.)--University of California, Irvine},
   publisher={ProQuest LLC, Ann Arbor, MI},
   date={2019},
   pages={83},
   isbn={978-1687-92504-6},
   review={\MR{4051258}},
}
%
%\bib{LeBrun}{article}{
%   author={LeBrun, Claude},
%   title={Diffeomorphisms, symplectic forms and Kodaira fibrations},
%   journal={Geom. Topol.},
%  volume={4},
%   date={2000},
%   pages={451--456},
%   issn={1465-3060},
%   review={\MR{1796448}},
%   doi={10.2140/gt.2000.4.451},
%}
%
\bib{McT}{article}{
   author={McMullen, C. T.},
   author={Taubes, C. H.},
   title={4-manifolds with inequivalent symplectic forms and 3-manifolds
   with inequivalent fibrations},
   journal={Math. Res. Lett.},
   volume={6},
   date={1999},
   number={5-6},
   pages={681--696},
   issn={1073-2780},
   review={\MR{1739225}},
   doi={10.4310/MRL.1999.v6.n6.a8},
}
%
\bib{Neumann}{article}{
   author={Neumann, W. D.},
   title={Splicing algebraic links},
   conference={
      title={Complex analytic singularities},
   },
   book={
      series={Adv. Stud. Pure Math.},
      volume={8},
      publisher={North-Holland, Amsterdam},
   },
   date={1987},
   pages={349--361},
   review={\MR{894301}},
}
%
%\bib{Ruan}{article}{
%   author={Ruan, Y.},
%  title={Symplectic topology on algebraic $3$-folds},
%   journal={J. Differential Geom.},
%   volume={39},
%   date={1994},
%   number={1},
%   pages={215--227},
%   issn={0022-040X},
%   review={\MR{1258920}},
%}
%
%\bib{Smith}{article}{
%   author={Smith, Ivan},
%  title={On moduli spaces of symplectic forms},
%   journal={Math. Res. Lett.},
%   volume={7},
%   date={2000},
%   number={5-6},
%   pages={779--788},
%   issn={1073-2780},
%   review={\MR{1809301}},
%   doi={10.4310/MRL.2000.v7.n6.a10},
%}
%
%\bib{Thurston}{article}{
%   author={Thurston, W. P.},
%   title={Some simple examples of %symplectic manifolds},
%   journal={Proc. Amer. Math. Soc.},
%  volume={55},
%   date={1976},
%   number={2},
%   pages={467--468},
%   issn={0002-9939},
%   review={\MR{402764}},
%   doi={10.2307/2041749},
%}
%
\bib{TsengIII}{article}{
   author={Tsai, C.-J.},
   author={Tseng, L.-S.},
   author={Yau, S.-T.},
   title={Cohomology and Hodge theory on symplectic manifolds: III},
   journal={J. Differential Geom.},
   volume={103},
   date={2016},
   number={1},
   pages={83--143},
   issn={0022-040X},
   review={\MR{3488131}},
}
%
\bib{TsengII}{article}{
   author={Tseng, L.-S.},
   author={Yau, S.-T.},
   title={Cohomology and Hodge theory on symplectic manifolds: II},
   journal={J. Differential Geom.},
   volume={91},
   date={2012},
   number={3},
   pages={417--443},
   issn={0022-040X},
   review={\MR{2981844}},
}
%
\bib{Graph_Link}{article}{
   author={Vidussi, S.},
   title={Homotopy $K3$'s with several symplectic structures},
   journal={Geom. Topol.},
   volume={5},
   date={2001},
   pages={267--285},
   issn={1465-3060},
   review={\MR{1825663}},
   doi={10.2140/gt.2001.5.267},
}
%
\bib{Vidussi}{article}{
   author={Vidussi, S.},
   title={Smooth structure of some symplectic surfaces},
   journal={Michigan Math. J.},
   volume={49},
   date={2001},
   number={2},
   pages={325--330},
   issn={0026-2285},
   review={\MR{1852306}},
   doi={10.1307/mmj/1008719776},
}
\end{biblist}
\end{bibdiv}
\end{document}